\documentclass[a4paper,11pt]{amsart}
\usepackage{amsfonts,stmaryrd,amssymb,amsmath,a4wide,verbatim}
\usepackage[british]{babel}
\usepackage{showkeys}
\usepackage[active]{srcltx}\ifx\pdfoutput\undefined
    \usepackage[pdftex]{graphicx}
    \DeclareGraphicsExtensions{.jpg, .png , .pdf, .eps}
\else
    \usepackage[dvips]{graphicx}
    \DeclareGraphicsExtensions{.eps, .ps, .eps.gz, ps.gz}
\fi

\makeatletter
\@addtoreset{equation}{section}\makeatother

\newcommand{\scal}[2]{\left\langle #1 , #2\right\rangle}
\newcommand{\scalpoly}[2]{\left( #1 , #2\right)_{{\S}^n}}

\newcommand{\normdpoly}[1]{\left\| #1\right\|_{\S^n}}
\newcommand{\normdfonc}[1]{\| #1\|_2}

\def\tr{{\rm tr}\,}

\newtheorem{theorem}{Theorem}[section]
\newtheorem{example}[theorem]{Example}
\newtheorem{lemma}[theorem]{Lemma}

\newtheorem{proposition}[theorem]{Proposition}
\newtheorem{remark}[theorem]{Remark}

\def\R{\mathbb{R}}
\def\N{\mathbb{N}}
\def\S{\mathbb{S}}
\def\insm{\displaystyle\int_{M}}
\def\inssn{\int_{\S^n}}
\def\vol{dv}
\def\voleps{dv_{g_{\varepsilon}}}
\def\Vol{{\rm Vol}\,}
\def\H{{\rm H}}
\def\B{{\rm B}}
\def\xb{\overline{X}}

\def\la{\lambda_1}
\def\can{\text{can\ }}

\def\muks{\mu_k}
\def\hkm{{\mathcal H}^k(M)}
\def\hkr{{\mathcal H}^k(\R^{n+1})}
\def\coef{\fieps'^2+(1\pm\fieps)^2}
\def\coeff{\varphi'^2+(1+\varphi)^2}
\def\fieps{\varphi_{\varepsilon}}

\def\fiteps{\tilde{\varphi}_{\varepsilon}}
\def\mpeps{M_{\varepsilon}^+}
\def\mmeps{M_{\varepsilon}^-}
\def\meps{M_{\varepsilon}}

\def\ieps{I_{\varepsilon}}
\def\xeps{X_{\varepsilon}}
\def\Heps{H_{\varepsilon}}
\def\heps{h^{\pm}_{\varepsilon}}
\def\beps{\B _{\varepsilon}}
\def\hepsi{h^{\pm}_{1,\varepsilon}}
\def\hepsd{h^{\pm}_{2,\varepsilon}}
\def\hepst{h^{\pm}_{3,\varepsilon}}

\begin{document}
\title[]{Spectrum of hypersurfaces with small extrinsic radius or large $\lambda_1$ in Euclidean spaces}



\author[E. AUBRY, J.-F. GROSJEAN]{Erwann AUBRY, Jean-Fran\c cois GROSJEAN}

\address[E. Aubry]{LJAD, Universit\'e de Nice Sophia-Antipolis, CNRS; 28 avenue Valrose, 06108 Nice, France}
\email{eaubry@unice.fr}

\address[J.-F. Grosjean]{Institut \'Elie Cartan de Lorraine (Math\'ematiques), Universit\'e de Lorraine, B.P. 239, F-54506 Vand\oe uvre-les-Nancy cedex, France}
\email{jean-francois.grosjean@univ-lorraine.fr}

\date{\today}

\begin{abstract} In this paper, we prove that Euclidean hypersurfaces with almost extremal extrinsic radius or $\lambda_1$ have a spectrum that asymptotically contains the spectrum of the extremal sphere in the Reilly or Hasanis-Koutroufiotis Inequalities. We also consider almost extremal hypersurfaces which satisfy a supplementary bound on $v_M\|\B\|_\alpha^n$ and show that their spectral and topological properties depends on the position of $\alpha$ with respect to the critical value $\dim M$. The study of the metric shape of these extremal hypersurfaces will be done in \cite{AG1}, using estimates of the present paper. 
\end{abstract}

\maketitle

\section{Introduction}

Throughout the paper, $X{:}\,M^n\to\R^{n+1}$ is a closed, connected, immersed Euclidean hypersurface (with $n\geqslant 2)$. We set $v_M$ its volume, $\B$ its second fundamental form, $\H=\frac{1}{n}\tr \B$ its mean curvature, $r_{M}$ its extrinsic radius (i.e. the least radius of the Euclidean balls containing $M$), $(\lambda_i^M)_{i\in\N}$ the non-decreasing sequence of its eigenvalues labelled with multiplicities and $\xb:=\frac{1}{v_M}\int_MX\vol$. For any function $f:M\to\R$, we set $\|f\|_\alpha^\alpha=\frac{1}{v_M}\int_M|f|^\alpha\vol$.
\bigskip

The Hasanis-Koutroufiotis inequality asserts that
\begin{equation}\label{rext}
 r_{M}\|\H\|_2\geqslant1,
\end{equation}
with equality if and only if $M$ is the Euclidean sphere $S_M$  with center $\overline{X}$ and radius $\frac{1}{\|\H\|_2}$.

The Reilly inequality asserts that
\begin{equation}\label{lambda}
\la^M\leqslant n\|\H\|_2^2,
\end{equation}
once again with equality if and only if $M$ is the sphere $S_M$ (we give some short proof of these inequalities in section \ref{prel}). 

Our aim is to study the spectral properties of the hypersurfaces that are almost extremal for each of this Inequalities. The results and estimates of this paper are used in \cite{AG1} to study the metric shape of the almost extremal hypersurfaces.

We set  $\mu_k^{S_M}=k(n{+}k{-}1)\|\H\|_2^2$ the $k$-th eigenvalue of $S_M$ (labelled without multiplicities) and $m_k$ its multiplicity.
Throughout the paper we shall adopt the notation that $\tau(\varepsilon|n,\cdots)$ is a positive function which depends on $n,\cdots$
and which converges to zero with $\varepsilon\to 0$ when $n,\cdots$ are fixed. 
 
\begin{theorem}\label{maintheo}
For any immersed hypersurface $M\hookrightarrow\R^{n+1}$ with $r_{M}\|\H\|_2\leqslant1+\varepsilon$ (or with $\frac{n\|\H\|_2^2}{\lambda_1^M}\leqslant 1+\varepsilon$) and for any $k\leqslant\frac{1}{\tau(\varepsilon|n)}$  the interval 
$[(1-\tau(\varepsilon|n))\mu^{S_M}_k,(1+\tau(\varepsilon|n))\mu_k^{S_M}]$
contains at least $m_k$ eigenvalues of $M$ counted with multiplicities.
\end{theorem}

Note that by Theorem \ref{maintheo}, almost extremal hypersurfaces for the Reilly inequality must satisfy 
$\frac{n\|\H\|_2^2}{1+\varepsilon}\leqslant\lambda_1^M\leqslant\cdots\leqslant\lambda_{n+1}^M\leqslant\bigl(1+\tau(\varepsilon|n)\bigr)n\|\H\|_2^2$
and so must have at least $n+1$ eigenvalues close to $\lambda_1^{S_M}=n\|\H\|_2^2$. This is very different from the almost extremal manifolds for the Lichnerowicz Inequality in positive Ricci curvature (see \cite{Au}).
\medskip

The proof of Theorem \ref{maintheo} is based on estimates for the restrictions to $M$ of homogeneous, harmonic polynomials of the ambient space $\R^{n+1}$. Such a polynomial of degree $k$ satisfies the equality $\Delta^{S_M}P=n\|\H\|_2^2dP(X)+\|\H\|_2^2D^0dP(X,X)=\mu_k^{S_M}P$
whereas its restriction on $M$ satisfies
$\Delta^MP=n\H dP(\nu)+D^0dP(\nu,\nu)$
where $D^0dP$ is the Euclidean Hessian and $\nu$ a local unit, normal vector to $M$. We prove that on almost extremal hypersurfaces, the quantities $\nu-\H X$ and $|\H|-\|\H\|_2$ are small in $L^2$-norms, which, by careful computations, gives essentially the following estimates (see Lemmas \ref{almosteigenf} and \ref{Ppresquortho2})
\begin{align}
\bigl|\|\varphi P\|_{L^2(M)}^2-\|\varphi P\|^2_{L^2(S_M)}\bigr|&\leqslant\tau(\varepsilon|n,k)\|\varphi P\|_{L^2(S_M)},\\
\|\Delta^M\varphi P-\mu_k^{S_M}\varphi P\|_{L^2(M)}&\leqslant\tau(\varepsilon|n,k)\|\varphi P\|_{L^2(M)},
\end{align}
where $\varphi$ is a cut function localized near $S_M$.
The main difficulty in proving this estimate is that there is no known good local control of the measure on $M$ involving only the $L^2$-norm of the mean curvature.
\medskip

Theorem \ref{maintheo} does not say that the spectrum of almost extremal hypersurfaces is close to the spectrum of $S_M$, but only that the spectrum of $S_M$ asymptotically appears in the spectrum of $M$. Our next result shows that it is optimal in dimension larger than $2$, even under a supplementary (not too strong) bound on the sectional curvature.

\begin{theorem}\label{ctrexple4}
 Let $M_1,M_2\hookrightarrow\R^{n+1}$ be two immersed compact submanifolds of dimension $m\geqslant 3$, $M_1\#M_2$ be their connected sum and $F$ be any closed subset of $]0,+\infty[$ containing ${\rm Sp}(M_1)\setminus\{0\}$. Then there exists a sequence of immersions $i_k:M_1\#M_2\hookrightarrow\R^{n+1}$ such that\\

1) $i_k(M_1\# M_2)$ converges to $M_1$ in Hausdorff topology,

2) the curvatures of $i_k(M_1\# M_2)$ satisfy
\begin{align*}
&\int_{i_k(M_1\#M_2)}|\H|^\alpha\to \int_{M_1}|\H|^\alpha\quad\quad\mbox{for any }\alpha\in[1,m),\\
&\int_{i_k(M_1\#M_2)}|\B|^\alpha\to \int_{M_1}|\B|^\alpha\quad\quad\mbox{for any }\alpha\in[1,m),
\end{align*}

3) $\cap_{k\in\N}\overline{\cup_{l\geqslant k}{\rm Sp}\bigl(i_l(M_1\#M_2)\bigr)}=F\cup\{0\}$,

4) $\Vol(i_k(M_1\#M_2))\to\Vol M_1$.
\end{theorem}

To get almost extremal submanifolds from the previous result, we just have to consider the case where $M_1=\S^n$ (and $F\subset[n,+\infty[$). It gives almost extremal hypersurfaces for the Reilly or Hasanis-Koutroufiotis Inequalities with the topology of any immersible Euclidean hypersurface, a spectrum as Hausdorff-close as we want of any closed set containing ${\rm Sp}(\S^n)$ (and contained in $[n,+\infty[$), even if we assume a bound on $v_M\|\B\|_\alpha^n$ for any $\alpha<n$.

On the other hand, if we assume a bound on $\|\B\|_\alpha$ with $\alpha>n$, we prove in \cite{AG1} that the almost extremal hypersrfaces converge to $S_M$ in Hausdorff distance, which combined with the $\mathcal{C}^{1,\beta}$ pre-compactness theorem of \cite{Dell} (or a Moser iteration as in the previous version of this paper \cite{AGR1}) imply the following stability in Lipschitz distance.

\begin{theorem}\label{Lipschitz} Let $n<\alpha\leqslant\infty$. Any immersed hypersurface $M\hookrightarrow \R^{n+1}$ with $v_M\|\B \|_\alpha^n\leqslant A$ and $r_{M}\|\H\|_2\leqslant1+\varepsilon$ (or with $v_M\|\B \|_\alpha^n\leqslant A$ and $\frac{n\|\H\|_2^2}{\lambda_1}\leqslant 1+\varepsilon $) is diffeomorphic to $S_M$ and satisfies $d_L(M,S_M)\leqslant\tau(\varepsilon|n,\alpha,A)$. In particular, we have $|\lambda^M_k-\lambda_k^{S_M}|\leqslant\tau(\varepsilon|k,n,\alpha,A)$ for any $k\in\N$.
\end{theorem}

Eventually, the critical case where we assume an upper bound on $v_M\|\B\|_n^n$ will be studied in a forthcoming, but we construct in the present paper some examples of almost extremal hypersurfaces satisfying such a bound as a preliminary. First of all, considering the constructions of Theorem \ref{ctrexple4} in the case $\alpha=n$, we get almost extremal hypersurfaces for the two inequalities with the topology of any immersible hypersurface, with $v_M\|B\|_n^n$ bounded and whose spectrum asymptotically contains ${\rm Sp}(S_M)$ and a finite subset of $\R\setminus{\rm Sp}(S_M)$ (see section \ref{pot2}). Note however that the bound on $v_M\|\B\|_n^n$ will depend on the topology of the extremal hypersurfaces and on the values and number of their eigenvalues not close to ${\rm Sp}(S_M)$.

In section \ref{pot3}, we construct almost extremal hypersurfaces for the Hasanis-Koutroufiotis inequality, not diffeomorphic to $\S_M$, not Gromov-Hausdorff close to $S_M$, with limit spectrum larger than the spectrum of $\S^n$ and with $\|\H\|_\infty$ bounded. We set $E(x)$ the integral part of $x$.

\begin{example}\label{ctrexple1}
For any couple $(l,p)$ of integers there exists a sequence of embedded hypersurfaces $M_j\hookrightarrow \R^{n+1}$ diffeomorphic to $p$ spheres $\S^n$ glued by connected sum along $l$ points, such that $\|\H_j\|_{\infty}^{~}\leqslant C(n)$, $\|\H\|_2=1$, $\|\B_j\|_{n}^{~}\leqslant C(n)$, $\bigl\||X_j|-1\bigr\|_\infty\to0$, $\bigl\||\H_j|-1\bigr\|_1\to 0$, and for any $\sigma\in\N$ we have
$\lambda_\sigma^{M_j}\to\lambda_{E(\frac{\sigma}{p})}^{\S^n}.$
In particular, the $M_j$ have at least $p$ eigenvalues close to $0$ whereas its extrinsic radius is close to $1$.
\end{example}

\begin{example}\label{ctrexple3}
There exists sequence of immersed hypersurfaces $M_j\hookrightarrow\R^{n+1}$ diffeomorphic to $2$ spheres $\S^n$ glued by connected sum along $1$ great subsphere $\S^{n-2}$, such that $\|\H_j\|_{\infty}^{~}\leqslant C(n)$, $\|\H_j\|_2=1$, $\|\B_j\|_2^{~}\leqslant C(n)$, $\bigl\||X_j|-1\bigr\|_\infty\to0$, $\bigl\||\H_j|-1\bigr\|_1\to 0$, and for any $\sigma\in\N$ we have
$\lambda_\sigma^{M_j}\to\lambda_{E(\frac{\sigma}{2})}^{\S^{n,d}},$
where $\S^{n,d}$ is the sphere $\S^n$ endowed with the singular metric, pulled-back of the canonical metric of $\S^n$ by the map $\pi:(y,z,r)\in\S^1\times\S^{n-2}\times[0,\frac{\pi}{2}]\mapsto(y^d,z,r)\in\S^1\times\S^{n-2}\times[0,\frac{\pi}{2}]$, where $\S^1\times\S^{n-2}\times[0,\frac{\pi}{2}]$ is identified with $\S^n\subset \R^{2}\times\R^{n-1}$ via the map $\Phi(y,z,r)=\bigl((\sin r) y,(\cos r) z\bigr)$.
Note that $\S^{n,d}$ has infinitely many eigenvalues that are not eigenvalues of $\S^n$.
\end{example}

\bigskip

The structure of the paper is as follows: after a preliminary section \ref{prel}, where we give short proofs of the Reilly and Hasanis-Koutroufiotis inequalities, we prove some concentration properties for the volume, mean curvature and position vector $X$ of almost extremal hypersurfaces in Section \ref{concentration}. Section \ref{Homog} is devoted to estimates on the restriction on hypersurfaces of the homogeneous, harmonic polynomials of $\R^{n+1}$. These estimates are used in Section \ref{potm} to prove Theorem \ref{maintheo}. We end the paper in section \ref{se} by the constructions of Theorem \ref{ctrexple4} and of Examples \ref{ctrexple1} and \ref{ctrexple3}.

 Throughout the paper we adopt the notation that $C(n,k,p,\cdots)$ is function greater than $1$ which depends on $p$, $q$, $n$, $\cdots$. It eases the exposition to disregard the explicit nature of these functions. The convenience of this notation is that even though $C$ might change from line to line in a calculation it still maintains these basic features.
\bigskip

\noindent\underline{Acknowledgments}:
Part of this work was done while E.A was invited at the MSI, ANU Canberra, funded by the PICS-CNRS Progress in Geometric Analysis and Applications. E.A. thanks P.Delanoe, J.Clutterbuck and J.X. Wang for giving him this opportunity.

\section{Some geometric optimal inequalities}\label{prel}

Any function $F$ on $\R^{n+1}$ gives rise to a function $F\circ X$ on $M$ which, for more convenience, will be also denoted $F$ subsequently. An easy computation gives the formula
\begin{equation}
 \label{fondhess}
 \Delta F=n\H dF(\nu)+\Delta^0F+\nabla^0dF(\nu,\nu),
\end{equation}
where $\nu$ denotes a local normal vector field of $M$ in $\R^{n+1}$, $\nabla^0$ is the Euclidean connection, $\Delta$ denotes the Laplace operator of $(M,g)$ and $\Delta^0$ is  the Laplace operator of $\R^{n+1}$. This formula is fundamental to control the geometry of a hypersurface by its mean curvature. Applied to $F(x)=\langle x,x\rangle$, where $\langle\cdot\, ,\cdot\rangle$ is the canonical product on $\R^{n+1}$, Formula \ref{fondhess} gives the Hsiung formulae,
\begin{align}\label{hsiung}\frac{1}{2}\Delta |X|^2=n\H\scal{\nu}{X}-n,\ \ \ \ \ \ \ \ \int_M\H\langle \nu,X\rangle\vol=v_M\end{align}

\subsection{A rough geometrical bound}\label{rbog}

The integrated Hsiung formula \eqref{hsiung} and the Cauchy-Schwarz inequality give the following
\begin{align}\label{prim2}1=\int_M\frac{ \H\langle \nu,X\rangle\vol}{v_M}\leqslant\|\H\|_2\bigl\|X-\xb\bigr\|_2\end{align}
 This inequality $\|\H\|_2\|X-\overline{X}\|_2\geqslant 1$ is optimal since $M$ satisfies 
$
\|\H\|_2\bigl\|X-\xb\bigr\|_2=1
$
if and only if $M$ is a sphere of radius $\frac{1}{\|\H\|_2}$ and center $\xb$. Indeed, in this case $X-\xb$ and $\nu$ are collinear on $M\setminus\{H=0\}$, hence $|X-\xb|^2$ is locally constant on $M\setminus\{H=0\}$. This implies that $\{H=0\}=\emptyset$ and that $X$ is an isometric-cover of $M$ on the sphere $S$ of center $\overline{X}$ and radius $\|X-\bar{X}\|_2=\frac{1}{\|\H\|_2}$, hence an isometry.

\subsection{Hasanis-Koutroufiotis inequality on extrinsic radius}

We set $R$ the extrinsic Radius of $M$, i.e. the least radius of the balls of $\R^{n+1}$ which contain $M$. Then Inequality \eqref{prim2} gives
$  \|\H\|_2^{~}r_{M}=\|\H\|_2^{~}\inf_{u\in\R^{n+1}}\|X-u\|_\infty\geqslant\|\H\|_2\inf_{u\in\R^{n+1}}\|X-u\|_2=\|\H\|_2\|X-\overline{X}\|_2\geqslant 1$
and $r_{M}=\frac{1}{\|\H\|_2}$ if and only if we have equality in \eqref{prim2}.

\subsection{Reilly inequality on $\mathbf{\lambda_1^M}$}

Since we have $\frac{1}{v_M}\int_M (X_i-\bar{X_i})\,dv=0$ for any component function of $X-\bar{X}$, by the min-max principle and Inequality \eqref{prim2}, we have
$\lambda_1^M\frac{1}{\|\H\|_2^2}\leqslant\lambda_1^M\|X-\bar{X}\|_2^2=\lambda_1^M\sum_i\|X_i-\bar{X_i}\|_2^2\leqslant\sum_i\|\nabla X_i\|_2^2=n$
where $\lambda_1^M$ is the first non-zero eigenvalue of $M$ and where the last equality comes from the fact that $\sum_i|\nabla X_i|^2$ is the trace of the quadratic form $Q(u)=|p(u)|^2$ with respect to the canonical scalar product, where $p$ is the orthogonal projector from $\R^{n+1}$ to $T_xM$. This gives the Reilly inequality \eqref{lambda}.

Here also, equality in the Reilly inequality gives equality in $\ref{prim2}$ and so it characterizes the sphere of radius $\frac{1}{\|\H\|_2}=\|X\|_2=\sqrt{\frac{n}{\lambda_1^M}}$.


\section{Concentration estimates}\label{concentration}

We say that $M$ satisfies the pinching $(P_{p,\varepsilon})$ when $\|\H\|_p\|X-\overline{X}\|_2\leqslant 1+\varepsilon$.
From the proofs of Inequalities \eqref{rext} and \eqref{lambda} above, it appears that pinchings $r_{M}\|\H\|_2\leqslant1+\varepsilon$ or $n\|\H\|_2^2/\lambda_1\leqslant1+\varepsilon$ imply the pinching $(P_{2,\varepsilon})$. 

From now on, we assume, without loss of generality, that $\bar{X}=0$.
Let $X^T(x)$ denote the orthogonal projection of $X(x)$ on the tangent space $T_xM$.

\begin{lemma}\label{banalite} If $(P_{2,\varepsilon})$  holds, then we have $\|X^T\|_2\leqslant\sqrt{3\varepsilon}\|X\|_2$ and $\|X-\frac{\H}{\|\H\|_2^2}\nu\|_2\leqslant\sqrt{3\varepsilon}\|X\|_2$.
\end{lemma}

\begin{proof}
Since we have $1=\frac{1}{v_M}\int_M\H\langle X,\nu\rangle\vol\leqslant\|\H\|_2\|\langle X,\nu\rangle\|_2$, Inequality $(P_{2,\varepsilon})$ gives us
$\|X\|_2\leqslant(1+\varepsilon)\|\langle X,\nu\rangle\|_2$ and $1\leqslant\|\H\|_2\|X\|_2\leqslant1+\varepsilon$. Hence $\|X-\langle X,\nu\rangle\nu\|_2^{~}\leqslant\sqrt{3\varepsilon}\,\|X\|_2^{~}$ and $\|X-\frac{\H\nu}{\|\H\|_2^2}\|_2^{2}=\|X\|_2^2-\|\H\|_2^{-2}\leqslant3\varepsilon\,\|X\|_2^2$.
\end{proof}

We set $A_\eta=B_0(\frac{1+\eta}{\|\H\|_2})\setminus B_0(\frac{1-\eta}{\|\H\|_2})$.

\begin{lemma}\label{estimplus}
If $(P_{p,\varepsilon})$ (for $p>2$), or $ n\|\H\|_2^2/\lambda_1^M\leqslant1+\varepsilon$, or $r_{M}\|\H\|_2\leqslant1+\varepsilon$ holds (with $\varepsilon\leqslant\frac{1}{100}$), then we have
$\bigl\|\|X\|-\frac{1}{\|\H\|_2}\bigr\|_2\leqslant\frac{C}{\|\H\|_2}\sqrt[8]{\varepsilon}$, $\||\H|-\|\H\|_2\|_2\leqslant C\sqrt[8]{\varepsilon} \|\H\|_2$
and
$\Vol (M\setminus A_{\sqrt[8]{\varepsilon}})\leqslant C\sqrt[8]{\varepsilon}v_M$,
where $C=6\times2^\frac{2p}{p-2}$ in the case ($P_{p,\varepsilon}$) and $C=100$ in the other cases. 
\end{lemma}

\begin{proof}
When $(P_{p,\varepsilon})$ holds, we have
$$\|\H\|_p\|X\|_2\leqslant(1+\varepsilon)\leqslant(1+\varepsilon)\|\H\|_p\|X\|_\frac{p}{p-1}\leqslant(1+\varepsilon)\|\H\|_p\|X\|_1^{1-\frac{2}{p}}\|X\|_2^\frac{2}{p},$$
 hence we get $\bigl\||X|-\frac{1}{\|\H\|_2}\bigr\|_2^2=\|X\|_2^2-2\frac{\|X\|_1}{\|\H\|_2}+\frac{1}{\|\H\|_2^2}\leqslant2^\frac{2p}{p-2}\frac{1}{\|\H\|_2^2}\varepsilon$.
Combined with the second inequality of Lemma \ref{banalite}, it gives
\begin{align*}
\bigl\||\H|-\|\H\|_2\bigr\|_2&\leqslant\|\H\|_2^2\bigl\||X|-\frac{|\H|}{\|\H\|_2^2}\bigr\|_2+\|\H\|_2^2\bigl\||X|-\frac{1}{\|\H\|_2}\bigl\|_2\leqslant C\sqrt[4]{\varepsilon}\|\H\|_2
\end{align*}
Now, by the Chebyshev inequality and Lemma \ref{banalite}, we get
\begin{align*} \Vol\bigl(M\setminus A_{\sqrt[4]{\varepsilon} }\bigr)&=\Vol\Bigl\{x\in M/\,\bigl||X(x)|-\frac{1}{\|\H\|_2}\bigr|\geqslant\frac{\sqrt[4]{\varepsilon}}{\|\H\|_2}\Bigr\}\\
&\leqslant\frac{\|\H\|_2^2}{\sqrt{\varepsilon}}\int_M\bigl||X|-\frac{1}{\|\H\|_2}\bigr|^2\leqslant C(p)\sqrt{\varepsilon}v_M
\end{align*}

When $r_{M}\|\H\|_2^{~}\leqslant1+\varepsilon$ holds. We set $X_0$ the center of the circumsphere to $M$ of radius $r_{M}$. We have $\|X-X_0\|_2^2=\|X\|_2^2+|X_0|^2=r_{M}^2\leqslant\frac{(1+\varepsilon)^2}{\|\H\|_2^2}$ and then we have $|X_0|\leqslant\frac{\sqrt{3\varepsilon}}{\|\H\|_2}$ and $|X|\leqslant|X_0|+r_{M}\leqslant\frac{1+3\sqrt{\varepsilon}}{\|\H\|_2}$. So we have $\frac{1}{\|\H\|_2^2}-|X|^2\in[\frac{\sqrt[4]{\varepsilon}}{\|\H\|_2^2},\frac{1}{\|\H\|_2}]$ on $M\setminus A_{\sqrt[4]{\varepsilon}}$. Chebyshev  inequality and \eqref{prim2} give us
\begin{align*}
\frac{\Vol(M\setminus A_{\sqrt[4]{\varepsilon}})}{v_M}\frac{\sqrt[4]{\varepsilon}}{\|\H\|_2^2}&\leqslant\frac{1}{v_M}\int_{M\setminus A_{\sqrt[4]{\varepsilon}}}\frac{1}{\|\H\|_2^2}-|X|^2\leqslant\frac{1}{v_M} \int_{M\cap A_{\sqrt[4]{\varepsilon}}}|X|^2-\frac{1}{\|\H\|_2^2}\leqslant\frac{9\sqrt{\varepsilon}}{\|\H\|_2^2}
\end{align*}
where in the last inequality we have used $|X|\leqslant\frac{1+3\sqrt{\varepsilon}}{\|\H\|_2}$ and, so we get
\begin{align*}
\bigl\||X|-\frac{1}{\|\H\|_2}\bigr\|_2^2&=\frac{1}{v_M}\int_{M\cap A_{\sqrt[4]{\varepsilon}}}\bigl||X|-\frac{1}{\|\H\|_2}\bigr|^2+\frac{1}{v_M}\int_{M\setminus A_{\sqrt[4]{\varepsilon}}}\bigl||X|-\frac{1}{\|\H\|_2}\bigr|^2\\
&\leqslant\frac{\sqrt{\varepsilon}}{\|\H\|^2_2}+\frac{\Vol(M\setminus A_{\sqrt[4]{\varepsilon}})}{v_M}\frac{1}{\|\H\|_2^2}\leqslant \frac{10\sqrt[4]{\varepsilon}}{\|\H\|_2^2}
\end{align*}
Combined with the second inequality of Lemma \ref{banalite}, we get $\|\frac{1}{\|\H\|_2}-\frac{|\H|}{\|\H\|_2^2}\|_2\leqslant\frac{C\sqrt[8]{\varepsilon}}{\|\H\|_2}$.

When $ n\|\H\|_2^2/\lambda_1^M\leqslant1+\varepsilon$ holds, we have $\int_M(|X|^2-\|X\|_2^2)dv=0$ and so by the Poincare inequality we get
$\bigl\||X|^2-\|X\|_2^2\bigr\|_2^2\leqslant\frac{4\|X^T\|_2^2}{\lambda_1^M}\leqslant\frac{12(1+\varepsilon)^2\varepsilon\|X\|_2^2}{n\|\H\|_2^2}\leqslant\frac{200\varepsilon}{n\|\H\|_2^4}$,
which gives $\frac{1}{\|\H\|_2}\bigl\||X|-\frac{1}{\|\H\|_2}\bigr\|_2\leqslant\bigl\||X|^2-\frac{1}{\|\H\|_2^2}\bigr\|_2\leqslant\bigl\||X|^2-\|X\|_2^2\bigr\|_2+\bigl|\|X\|_2^2-\frac{1}{\|\H\|_2^2}\bigr|\leqslant\frac{12\sqrt{\varepsilon}}{\|\H\|_2^2}
$
and then we get the estimate on the volume of $A_{\sqrt[4]{\varepsilon}}$ by the same Chebyshev procedure as for $P_{p,\varepsilon}$ and the estimate on the mean curvature by the same procedure as for $r_{M}\|\H\|_2\leqslant 1+\varepsilon$.
\end{proof}

Let $\psi{:}[0,\infty)\rightarrow [0,1]$ be a smooth function with $\psi{=}0$ outside $[\frac{(1-2\sqrt[16]{\varepsilon})^2}{\|\H\|_2^2},\frac{(1+2\sqrt[16]{\varepsilon})^2}{\|\H\|_2^2}]$ and $\psi{=}1$ on $[\frac{(1-\sqrt[16]{\varepsilon})^2}{\|\H\|_2^2},\frac{(1+\sqrt[16]{\varepsilon})^2}{\|\H\|_2^2}]$. Let us consider the function $\varphi$ on $M$ defined by $\varphi(x)=\psi(|X_x|^2)$ and the vector field $Z$ on $M$ defined by $Z=\nu-\H X$. The previous estimates then imply the following.

\begin{lemma}\label{estinormphiz} $(P_{p,\varepsilon})$ (for $p>2$) or $ n\|\H\|_2^2/\lambda_1\leqslant1+\varepsilon$ or $r_{M}\|\H\|_2\leqslant 1+\varepsilon$ implies $\|\varphi^2(\H^2-\|\H\|_2^2)\|_1\leqslant C\sqrt[8]{\varepsilon}\|\H\|_2^2$, $\|\varphi Z\|_2\leqslant C\varepsilon^\frac{3}{32}$ and $|\|\varphi\|_2^2-1|\leqslant C\sqrt[8]{\varepsilon}$,
 where $C$ is a constant which depends on $p$ in the case  $(P_{p,\varepsilon})$.
\end{lemma}

\begin{proof}
We have $\|\varphi^2(\H^2-\|\H\|_2^2)\|_1\leqslant\bigl\||\H|-\|\H\|_2\bigr\|_22\|\H\|_2\leqslant C\sqrt[8]{\varepsilon}\|\H\|^2_2$
and
\begin{align*}
\|\varphi Z\|_2^2=&\frac{1}{v_M}\int_M\varphi^2|Z|^2\vol=\frac{1}{v_M}\int_M\varphi^2(1-2\H\langle\nu,X\rangle+\H^2|X|^2)\vol\\
=&\frac{\|\H\|_2^2}{v_M}\int_M\varphi^2\bigl|X-\frac{\H}{\|\H\|_2^2}\nu\bigr|^2\vol+\frac{1}{\|\H\|_2^2v_M}\int_M(\|\H\|_2^2-\H^2)\varphi^2(1-|X|^2\|\H\|_2^2)\vol\\
\leqslant&\|\H\|_2^2\bigl\|X-\frac{\H}{\|\H\|_2^2}\nu\bigr\|_2^2+8\sqrt[16]{\varepsilon}\frac{\bigl\|\varphi^2(\H^2-\|\H\|_2^2)\bigr\|_1}{\|\H\|_2^2},
\end{align*}
which gives the result by Lemma \ref{banalite}.
Finally, we have $1-\frac{\Vol(M\setminus A_{\sqrt[8]{\varepsilon}})}{v_M}\leqslant\frac{\Vol (A_{\sqrt[8]{\varepsilon}}\cap M)}{v_M}\leqslant\|\varphi\|_2^2$ and $\|\varphi\|_2^2\leqslant 1$.
\end{proof}

\section{Homogeneous, harmonic polynomials of degree $k$}\label{Homog}

In this section, we give some estimates on harmonic homogeneous polynomials restricted to almost extremal hypersurfaces. They will be used subsequently to derive our result on the spectrum and on the volume of almost extremal manifolds. Let us begin by general estimates on harmonic, homogeneous polynomials.

\subsection{General estimates}

Let $\hkr$ be the space of homogeneous, harmonic polynomials of degree $k$ on $\R^{n+1}$. Note that $\hkr$ induces on $\S^n$ the spaces of eigenfunctions of $\Delta^{\S^n}$ associated to the eigenvalues $\mu_k:=k(n+k-1)$ with  multiplicity $m_k:=\begin{pmatrix}n+k-1\\ k\end{pmatrix}\displaystyle\frac{n+2k-1}{n+k-1}$. 

On the space $\hkr$, we set $\scalpoly{P}{Q}:=\frac{1}{\Vol\S^n}\inssn PQ dv_{\can},$
where $dv_{can}$ denotes the element volume of the sphere with its standard metric.

Remind that for any $P\in\hkr$ and any  $Y\in\R^{n+1}$, we have $dP(X)=kP(X)$ and $\nabla^0 dP(X,Y)=(k-1)dP(Y)$.

\begin{lemma}\label{Pcarre}
For any $x\in\R^{n+1}$ and $P\in \hkr$, we have $|P(x)|^2\leqslant\|P\|_{\S^n}^2m_k|x|^{2k}$. 
\end{lemma}

\begin{proof}  Let $(P_i)_{1\leqslant i\leqslant m_k}$ be an orthonormal basis of $\hkr$. For any $x\in\S^n$, $Q_x(P)=P^2(x)$ is a quadratic form on $\hkr$ whose trace is given by $\sum_{i=1}^{m_k}P_i^2(x)$. Since for any $x'\in\S^n$ and any $O\in O_{n+1}$ such that $x'=Ox$ we have $Q_{x'}(P)=Q_x(P\circ O)$ and since $P\mapsto P\circ O$ is an isometry of $\hkr$, we have $ \sum_{i=1}^{m_k}P_i^2(x)=\tr(Q_x)=\sum_{i=1}^{m_k}P_i^2(x')=\tr( Q_{x'})$. We infer that
$\sum_{i=1}^{m_k}\frac{1}{\Vol\S^n}\inssn P_i^2(x)\vol=m_k=\frac{1}{\Vol\S^n}\inssn\left(\sum_{i=1}^{m_k}P_i^2(x)\right)\vol$
and so $\sum_{i=1}^{m_k}P_i^2(x)=m_k$. By homogeneity of the $P_i$ we get 
\begin{equation}\label{P1}
\sum_{i=1}^{m_k}P_i^2(x)=m_k|x|^{2k},
\end{equation}
and by the Cauchy-Schwarz inequality applied to $P(x)=\sum_i(P,P_i)_{\S^n}P_i(x)$, we get the result.
\end{proof}

As an immediate consequence, we have the following lemma.

\begin{lemma}\label{grad}
For any $x,u\in\R^{n+1}$ and $P\in\hkr$, we have
$$|d_xP(u)|^2\leqslant\|P\|_{\S^n}^2m_k\Bigl(\frac{\muks}{n}|x|^{2(k-1)}|u|^2+\bigl(k^2-\frac{\muks}{n}\bigr)\langle u,x\rangle^2|x|^{2(k-2)}\Bigr).$$
\end{lemma}

\begin{proof} Let $x\in\S^n$ and $u\in\S^n$  so that $\langle u,x\rangle=0$. Once again the quadratic forms $Q_{x,u}(P)=\bigl(d_x P(u)\bigr)^2$ are conjugate (since $O_{n+1}$ acts transitively on orthonormal couples) and so $\displaystyle\sum_{i=1}^{m_k}\bigl(d_xP_i(u)\bigr)^2$ does not depend on $u\in x^{\perp}$ nor on $x\in\S^n$. By choosing an orthonormal basis $(u_j)_{1\leqslant j\leqslant n}$ of $x^\perp$, we obtain that
\begin{align*}\sum_{i=1}^{m_k}\bigl(d_x P_i(u)\bigr)^2&=\frac{1}{n}\sum_{i=1}^{m_k}\sum_{j=1}^n\bigl(d_xP_i(u_j)\bigr)^2=\frac{1}{n\Vol\S^n}\int_{\S^n}\sum_{i=1}^{m_k}|\nabla^{\S^n}P_i|^2\\
&=\frac{1}{n\Vol\S^n}\int_{\S^n}\sum_{i=1}^{m_k}P_i\Delta^{\S^n}P_i=\frac{m_k\muks}{n}
\end{align*}
Now suppose that $u\in\R^{n+1}$. Then $u=v+\langle u,x\rangle x$, where $v=u-\langle u,x\rangle x$, and we have
\begin{align*}\sum_{i=1}^{m_k}\bigl(d_xP_i(u)\bigr)^2&=\sum_{i=1}^{m_k}\bigl(d_xP_i(v)+k\langle u,x\rangle P_i(x)\bigr)^2\\
&=\sum_{i=1}^{m_k}\bigl(d_xP_i(v)\bigr)^2+2k\langle u,x\rangle\sum_{i=1}^{m_k}d_xP_i(v)P_i(x)+m_k\langle u,x\rangle^2 k^2\\
&=\frac{m_k\muks}{n}|v|^2+m_k\langle u,x\rangle^2 k^2=m_k\left(\frac{\muks}{n}|u|^2+\left(k^2-\frac{\muks}{n}\right)\langle u,x\rangle^2 \right),
\end{align*}
where we have taken the derivative the equality \eqref{P1} to compute ${\displaystyle\sum_{i=1}^{m_k}}d_xP_i(v)P_i(x)$. By homogeneity of $P_i$ we get
${\displaystyle\sum_{i=1}^{m_k}}\bigl(d_xP_i(u)\bigr)^2=m_k\bigl(\frac{\muks}{n}|x|^{2(k-1)}|u|^2+(k^2-\frac{\muks}{n})\langle u,x\rangle^2|x|^{2(k-2)}\bigr)
$
and conclude once again by the Cauchy-Schwarz inequality.
\end{proof}

\begin{lemma}\label{Hess} For any $x\in\R^{n+1}$ and $P\in\hkr$, we have 
$$|\nabla^0dP(x)|^2\leqslant\|P\|_{\S^n}^2m_k\alpha_{n,k}|x|^{2(k-2)},$$ where $\alpha_{n,k}=(k-1)(k^2+\mu_k)(n+2k-3)\leqslant C(n)k^4$.
\end{lemma}

\begin{proof} The Bochner equality gives
\begin{align}\label{P3}
\sum_{i=1}^{m_k}|\nabla^0dP_i(x)|^2&=\sum_{i=1}^{m_k}\left(\langle d\Delta^0 P_i,dP_i\rangle-\frac{1}{2}\Delta^0\bigl|dP_i\bigr|^2\right)\notag\\
&=-\frac{1}{2}m_k\bigl(k^2+\mu_k\bigr)\Delta^0|X|^{2k-2}=m_k\alpha_{n,k}|X|^{2k-4}
\end{align}
\end{proof}

\subsection{Estimates on hypersurfaces}

Let $\hkm=\{P\circ X\ ,\ P\in\hkr\}$ be the space of functions induced on $M$ by $\hkr$. We will identify $P$ and $P\circ X$ subsequently. There is no ambiguity since we have

\begin{lemma} Let $M^n$ be a compact manifold immersed by $X$ in $\R^{n+1}$ and let $(P_1,\ldots,P_m)$ be a linearly independent set of homogeneous polynomials of degree $k$ on $\R^{n+1}$. Then the set $(P_1\circ X,\ldots,P_m\circ X)$ is also linearly independent.
\end{lemma}

\begin{proof} Any homogeneous polynomial $P$ which is zero on $M$ is zero on the cone $\R^+{\cdot} M$. Since $M$ is compact there exists a point $x\in M$ so that $X_x\notin T_xM$ and so $\R^+{\cdot}M$ has non empty interior. Hence $P\circ X=0$ implies $P=0$.
\end{proof}

We now compare the $L^2$-norm of $P$ on $M$ with $L^2$-norm of $P$ on the sphere $S_M=\frac{1}{\|\H\|_2}\S^n$.
We still denote $\psi:[0,\infty)\longrightarrow[0,1]$ a smooth function which is $0$ outside $[\frac{(1-\eta)^2}{\|\H\|_2^2},\frac{(1+\eta)^2}{\|\H\|_2^2}]$, is $1$ on $[\frac{(1-\eta/2)^2}{\|\H\|_2^2},\frac{(1+\eta/2)^2}{\|\H\|_2^2}]$ and satisfies the upper bounds $|\psi'|\leqslant\frac{4\|\H\|_2^2}{\eta}$ and $|\psi''|\leqslant\frac{8\|\H\|_2^4}{\eta^2}$. We set $\varphi(x)=\psi(|X_x|^2)$ on $M$.  

\begin{lemma}\label{estimphi}
With the above restrictions on $\psi$ we have
$$\|\Delta\varphi^2\|_1\leqslant\frac{192\|\H\|_2^4}{\eta^2}\|X^T\|_2^2+\frac{16n\|\H\|_2^2}{\eta}\|\varphi Z\|_1$$
\end{lemma}

\begin{proof}
An easy computation yields that
\begin{align*}\Delta(\varphi^2)&=-(\psi^2)''(|X|^2)|d|X|^2|^2+(\psi^2)'(|X|^2)\Delta|X|^2\\
&=-4(\psi^2)''(|X|^2)|X^T|^2-2n(\psi^2)'(|X|^2)\scal{\nu}{Z}
\end{align*}
But the bound on the derivatives of $\psi$ gives us $|(\psi^2)'|\leqslant\frac{8\|\H\|_2^2}{\eta}\psi$ and $|(\psi^2)''|\leqslant\frac{48\|\H\|_2^4}{\eta^2}$.
Hence we get $\|\Delta\varphi^2\|_1\leqslant\frac{192\|\H\|_2^4}{\eta^2}\|X^T\|_2^2+\frac{16n\|\H\|_2^2}{\eta}\|\varphi Z\|_1$.
\end{proof}

\begin{lemma}\label{Ppresquortho}  Let $\varphi: M\to[0,1]$ be as above. There exists a constant $C=C(n)$ such that for any isometrically immersed hypersurface $M$ of $\R^{n+1}$ and any $P\in\hkm$, we have $\bigl|\|\H\|_2^{2k}\normdfonc{\varphi P}^2-\normdpoly{P}^2\bigr|\leqslant\Bigl(1-\|\varphi\|_2^2+ DC(n)\sum_{i=1}^km_i(1+\eta)^{2k}\Bigr)\normdpoly{P}^2$,
where $D=\|\varphi Z\|_2+\|\varphi Z\|_2^2+\frac{200\|\H\|_2^2}{\eta^2}\|X^\perp\|_2^2+\frac{16n}{\eta}\|\varphi Z\|_1+\frac{\|\varphi^2( \H^2-\|\H\|_2^2)\|_1}{\|\H\|_2^2}$.
\end{lemma}

\begin{proof} For any $P\in\hkm$ we have
\begin{align*}
&\|\varphi\nabla^0 P\|_2^2=\|\varphi dP(\nu)\|_2^2+\|\varphi dP\|_2^2\\
&=\|\varphi dP(Z)\|_2^2+k^2\|\varphi \H P\|_2^2+\frac{1}{v_M}\int_M\bigl(2k \H dP(\varphi Z)\varphi P+\varphi^2 P\Delta P-\frac{P^2\Delta(\varphi^2)}{2}\bigr)\vol
\end{align*}
Now, Formula \eqref{fondhess} applied to $P\in\hkr$ gives
\begin{equation}\label{laplap}\Delta P=\muks \H^2 P+(n+2k-2)\H dP(Z)+\nabla^0 dP(Z,Z)\end{equation}
hence, we get
\begin{align*}\|\varphi\nabla^0 P\|_2^2=&\|dP(\varphi Z)\|_2^2+(\mu_k+k^2)\|\H\varphi P\|_2^2\\
&+\frac{1}{v_M}\int_M\bigl(\varphi^2 P\nabla^0 dP(Z,Z)+(n+4k-2)\varphi \H dP(\varphi Z)P-\frac{P^2\Delta(\varphi^2)}{2}\bigr)\vol\\
=&\frac{1}{v_M}\int_M\Bigl((\mu_k+k^2)\bigl(\H^2-\|\H\|_2^2\bigr)\varphi^2 P^2+(n+4k-2)\H dP(\varphi Z)\varphi P\Bigr)\vol\\
&+\frac{1}{v_M}\int_M \bigr(P\nabla^0 dP(\varphi Z,\varphi Z)-\frac{P^2\Delta(\varphi^2)}{2}\bigr)\vol\\
&+(\muks+k^2)\|\H\|_2^2\|\varphi P\|_2^2+\|dP(\varphi Z)\|_2^2
\end{align*}
Now we have
\begin{align}\label{normgrad0}\normdpoly{\nabla^0 P}^2=\normdpoly{\nabla^{\S^n}P}^2+k^2\normdpoly{ P}^2=(\muks+k^2)\normdpoly{P}^2\end{align}
Hence 
$$\displaylines{
\|\H\|_2^{2k-2}\normdfonc{\varphi\nabla^0 P}^2-\normdpoly{\nabla^0 P}^2=(\muks+k^2)\bigl(\|\H\|_2^{2k}\|\varphi P\|_2^2-\normdpoly{P}^2\bigr)+\|\H\|_2^{2k-2}\normdfonc{dP(\varphi Z)}^2\hfill\cr
+\frac{\|\H\|_2^{2k-2}}{v_M}\int_M\varphi^2 P\Bigl((\mu_k+k^2)\bigl(\H^2-\|\H\|_2^2\bigr)P +\H (n+4k-2)dP(Z)+\nabla^0 dP(Z,Z)\Bigr)\vol\hfill\cr
-\frac{\|\H\|_2^{2k-2}}{v_M}\int_M\frac{P^2\Delta(\varphi^2)}{2}\hfill}$$
Which gives
\begin{align}\label{intermediaire}
\Bigl|\|&\H\|_2^{2k}\|\varphi P\|_2^2-\normdpoly{P}^2\Bigr|\\
\leqslant&\frac{1}{\muks+k^2}\Bigl|\|\H\|_2^{2k-2}\|\varphi\nabla^0 P\|_2^2-\normdpoly{\nabla^0 P}^2\Bigr|\notag\\
&+\frac{\|\H\|_2^{2k-2}}{\mu_k+k^2}\int_M\Bigl( (n+4k-2)|\H|\varphi|P||dP(\varphi Z)|+|dP(\varphi Z)|^2+|P||\nabla^0dP||\varphi Z|^2\Bigr)\notag\\
&+\frac{\|\H\|_2^{2k-2}}{v_M}\int_M\bigl(\varphi^2\bigl|\H^2-\|\H\|_2^{2}\bigr|P^2+\frac{P^2|\Delta(\varphi^2)|}{2}\bigr)\vol\notag
\notag\end{align}
By Lemma \ref{Pcarre}, we have
\begin{align*}
\frac{\|\H\|_2^{2k-2}}{v_M}\insm \bigl| \H^2-\|\H\|_2^{2}\bigr|(\varphi P)^2\vol&\leqslant\frac{m_k\normdpoly{P}^2\|\H\|_2^{2k-2}}{v_M}\insm\bigl|\varphi^2( \H^2-\|\H\|_2^{2})\bigr||X|^{2k}\vol\\
&\leqslant \normdpoly{P}^2m_k (1+\eta)^{2k}\frac{\|\varphi^2(\H^2-\|\H\|_2^2)\|_1}{\|\H\|_2^2}
\end{align*}
In the same way, we have
\begin{align*}
\frac{\|\H\|_2^{2k-2}}{v_M}\int_M\frac{P^2|\Delta(\varphi^2)|}{2}\vol\leqslant\normdpoly{P}^2 m_k(1+\eta)^{2k}\frac{\|\Delta(\varphi^2)\|_1}{\|\H\|_2^2}
\end{align*}
and using Lemma \ref{grad}, we get
\begin{align*}
\frac{\|\H\|_2^{2k-2}}{v_M}\int_M\varphi^2 |PdP(Z)\H|\vol&\leqslant \frac{m_k k\normdpoly{P}^2\|\H\|_2^{2k-2}}{v_M}\int_M\varphi^2|X|^{2k-1}|\H Z|\,dv\\
&\leqslant\normdpoly{P}^2m_kk(1+\eta)^{2k}\|\varphi Z\|_2
\end{align*}
and
\begin{align*}
\frac{\|\H\|_2^{2k-2}}{v_M}\int_M|dP(\varphi Z)|^2&\leqslant\|P\|_{\S^n}^2m_kk^2\frac{\|\H\|_2^{2k-2}}{v_M}\int_M|\varphi Z|^2|X|^{2(k-1)}\,dv\\
&\leqslant \normdpoly{P}^2m_kk^2(1+\eta)^{2k}\|\varphi Z\|_2^2
\end{align*}
Finally, using Lemma \ref{Hess}, we get
\begin{align*}
\frac{\|\H\|_2^{2k-2}}{v_M}\int_M|P||\nabla^0dP||\varphi Z|^2&\leqslant \|P\|_{\S^n}^2m_k\sqrt{\alpha_{n,k}}\frac{\|\H\|_2^{2k-2}}{v_M}\int_M|X|^{2(k-1)}|\varphi Z|^2\,dv\\
&\leqslant\normdpoly{P}^2m_k\sqrt{\alpha_{n,k}}(1+\eta)^{2k}\|\varphi Z\|_2^2
\end{align*}
which, combined with \eqref{intermediaire} and equation \eqref{normgrad0}, gives 
\begin{align*}
&\frac{\bigl|\|\H\|_2^{2k}\normdfonc{\varphi P}^2-\normdpoly{P}^2\bigr|}{\|P\|_{\S^n}^2}\leqslant\frac{\Bigl|\|\H\|_2^{2k-2}\normdfonc{\varphi\nabla^0 P}^2-\normdpoly{\nabla^0 P}^2\Bigr|}{\normdpoly{\nabla^0 P}^2}\\
&+C(n)m_k(1+\eta)^{2k}\Bigl(\|\varphi Z\|_2+\|\varphi Z\|_2^2+\frac{\|\Delta(\varphi^2)\|_1}{\|\H\|_2^2}+\frac{\|\varphi^2(\H^2-\|\H\|_2^2)\|_1}{\|\H\|_2^2}\Bigr)\\
&\leqslant\frac{\Bigl|\|\H\|_2^{2k-2}\normdfonc{\varphi\nabla^0 P}^2-\normdpoly{\nabla^0 P}^2\Bigr|}{\normdpoly{\nabla^0 P}^2}+C(n)m_k(1+\eta)^{2k}D
\end{align*}
In particular for $k=1$, we have $|\nabla^0 P|$ constant equal to $(1+n)\|P\|_{\S^n}^2$ and so
\begin{align*}\bigl|\|\H\|_2^{2}\normdfonc{\varphi P}^2-\normdpoly{P}^2\bigr|&\leqslant\bigl(1-\|\varphi\|_2^2+C(n)m_1(1+\eta)^2D\bigr)\normdpoly{P}^2
\end{align*}
Now, let $B_k=\sup\Bigl\{\frac{|\|\H\|_2^{2k}\normdfonc{\varphi P}^2-\normdpoly{P}^2|}{\normdpoly{P}^2} \ \mid \ P\in\hkr\setminus\{0\}\Bigr\}$. Then using that $\nabla^0P\in{\mathcal H}^{k-1}(\R^{n+1})$ and \eqref{normgrad0}, we get
\begin{align*}B_k\leqslant B_{k-1}+C(n)m_k(1+\eta)^{2k}D\leqslant 1-\|\varphi\|_2^2+C(n)D\sum_{i=1}^km_i(1+\eta)^{2k}
\end{align*}
\end{proof}


\section{Proof of Theorem \ref{maintheo}}\label{potm}

Under the assumption of Theorem \ref{maintheo} we can use Lemmas \ref{banalite} and \ref{estinormphiz} to improve the estimate in Lemma \ref{Ppresquortho} in the case $\eta=2\sqrt[16]{\varepsilon}$.
\begin{lemma}\label{Ppresquortho2}
For any isometrically immersed hypersurface $M\hookrightarrow\R^{n+1}$ with $r_{M}\|\H\|_2\leqslant1+\varepsilon$ (or $\lambda_1(1+\varepsilon)^2\geqslant n\|\H\|_2^2$ or $(P_{p,\varepsilon})$ for $p>2$) and for any $P\in\hkm$, we have
$$\bigl|\|\H\|_2^{2k}\normdfonc{\varphi P}^2-\normdpoly{P}^2\bigr|\leqslant C\sqrt[32]{\varepsilon}\normdpoly{P}^2,$$
where $C=C(n,k)$ in the first two cases and $C=C(p,k,n)$ in the latter case.
\end{lemma}

As a consequence, the map $P\mapsto \varphi P$ is injective on $\hkm$ for $\varepsilon$ small enough.
\begin{lemma}\label{dimension} Under the assumption of Lemma \ref{Ppresquortho2}, if $\varepsilon\leqslant\frac{1}{(2C)^{32}}$ then $\dim(\varphi\hkm)=m_k$.
\end{lemma}

Lemma \ref{Ppresquortho2} allows us to prove the following estimate on $\Delta P$.
\begin{lemma}\label{almosteigenf}
Under the assumptions of Lemma \ref{Ppresquortho2}, if $\varepsilon\leqslant\frac{1}{(2C)^{32}}$, then for any  $P\in\hkm$, we have $\bigl\|\Delta(\varphi P)-\mu_k^{S_M}\varphi P\bigr\|_2\leqslant C\sqrt[16]{\varepsilon}\mu_k^{S_M}\|\varphi P\|_2$
where $C=C(n,k)$ ($C=C(n,k,p)$ under the pinching $(P_{p,\varepsilon})$).
\end{lemma}

\begin{proof}
Let $P\in\hkm$. Using \eqref{fondhess} we have
\begin{align*}
\Delta(\varphi P)=&P\Delta\varphi-2\langle dP, d\varphi\rangle+\varphi\Delta P=P\Delta\varphi-2\langle dP,d\varphi\rangle+\varphi n\H dP(\nu)+\varphi\nabla^0 dP(\nu,\nu)\\
=&P\Delta\varphi-2\langle dP, d\varphi\rangle+\varphi \mu_k|H|\|\H\|_2P+\varphi (n+k-1)\frac{\H}{|\H|}\|\H\|_2dP(Z)\\
&+\varphi (n+k-1)\frac{\H}{|\H|}(|\H|-\|\H\|_2)dP(\nu)+\varphi\nabla^0 dP(\nu,Z)
\end{align*}
hence, we get
\begin{align}\label{esti1}
&\normdfonc{\Delta(\varphi P)-\mu_k\normdfonc{H}^2\varphi P}\leqslant\normdfonc{(\Delta\varphi) P}+2\normdfonc{\scal{d\varphi}{dP}}+\mu_k\normdfonc{(|H|-\normdfonc{H})\varphi P}\|\H\|_2\notag\\
&+(n{+}k{-}1)\|\H\|_2\normdfonc{\varphi|dP||Z|}+ (n{+}k{-}1)\bigl\|\varphi(|\H|-\|\H\|_2)dP(\nu)\bigr\|_2+\normdfonc{\varphi|\nabla^0 dP||Z|}
\end{align}
Let us estimate $\normdfonc{(\Delta\varphi) P}$.
\begin{align*}\normdfonc{(\Delta\varphi) P}^2&\leqslant\frac{1}{v_M}\insm(4|\psi''(|X|^2)||X^T|^2+2n|\psi'(|X|^2)||Z|)^2 P^2\vol\\
&\leqslant\frac{m_k}{v_M}\Bigl(\insm|X|^{2k}\bigl(4|\psi''(|X|^2)||X^T|^2+2n|\psi'(|X|^2)||Z|\bigr)^2\vol\Bigr)\normdpoly{P}^2\\
&\leqslant\frac{m_k}{v_M}\frac{(1+2\sqrt[16]{\varepsilon})^{2k}}{\normdfonc{H}^{2k}}\Bigl(\int_{A_{2\sqrt[16]{\varepsilon}}}\bigl(\frac{8\normdfonc{H}^4}{\sqrt[8]{\varepsilon}}|X^T|^2+2n\frac{2\normdfonc{H}^2}{\sqrt[16]{\varepsilon}}|Z|\bigr)^2\vol\Bigr)\normdpoly{P}^2\\ 
&\leqslant\frac{m_k}{v_M}\frac{(1+2\sqrt[16]{\varepsilon})^{2k}}{\normdfonc{H}^{2k}}\Bigl(\int_{A_{2\sqrt[16]{\varepsilon}}}\frac{128\normdfonc{H}^8}{\sqrt[4]{\varepsilon}}|X^T|^4+32n^2\frac{\normdfonc{H}^4}{\sqrt[8]{\varepsilon}}|Z|^2\vol\Bigr)\normdpoly{P}^2 
\end{align*}
Since we have $|X^T|\leqslant|X|$ and since Lemma \ref{estinormphiz} is valid with $\|\varphi Z\|_2^2$ replaced by $\frac{1}{v_M}\int_{A_{2\sqrt[16]{\varepsilon}}}|Z|^2$, we get
\begin{align*}\normdfonc{(\Delta\varphi) P}^2
&\leqslant\frac{C(n,k)\mu_k}{v_M}\frac{\normdpoly{P}^2}{\normdfonc{H}^{2k}}\int_{A_{2\sqrt[16]{\varepsilon}}}\bigl(\frac{\normdfonc{H}^6}{\sqrt[4]{\varepsilon}}|X^T|^2+\frac{\normdfonc{H}^4}{\sqrt[8]{\varepsilon}}|Z|^2\bigr)\vol\\
&\leqslant\frac{C(n,k)\mu_k}{\normdfonc{H}^{2k}}\normdfonc{H}^4\sqrt[16]{\varepsilon}\normdpoly{P}^2
\end{align*}

From the lemma \ref{Ppresquortho2}, $\varepsilon\leqslant \frac{1}{(2C)^{32}}$ implies that
\begin{align}\label{estiimportante}\normdpoly{P}^2\leqslant 2\normdfonc{H}^{2k}\normdfonc{\varphi P}^2\end{align}
 which gives
\begin{align}\label{esti2}\normdfonc{(\Delta\varphi) P}^2\leqslant C(n,k)\mu_k \normdfonc{H}^4\sqrt[16]{\varepsilon}\normdfonc{\varphi P}^2
\end{align}
Now
\begin{align}\label{esti3}
\normdfonc{\scal{d\varphi}{dP}}^2&\leqslant4\normdfonc{\psi'(|X|^2)|X^T||dP|}^2\leqslant\frac{16\normdfonc{\H}^4}{\sqrt[16]{\varepsilon} v_M}\int_{A_{2\sqrt[16]{\varepsilon}}}|X^T|^2|dP|^2\vol\notag\\
&\leqslant\frac{16\normdfonc{\H}^4}{\sqrt[16]{\varepsilon} v_M}\normdpoly{P}^2\int_{A_{2\sqrt[16]{\varepsilon}}}|X^T|^2m_knk^2|X|^{2(k-1)}\vol\notag\\
&\leqslant C(n,k)\mu_k\sqrt[16]{\varepsilon}\normdfonc{\H}^{4-2k}\normdpoly{P}^2\leqslant C(n,k)\normdfonc{\H}^4\sqrt[16]{\varepsilon}\normdfonc{\varphi P}^2
\end{align}
By the same way, we get
\begin{align}\label{esti7}
\|\varphi|dP|Z\|_2^2\leqslant C(n,k)\mu_k\|\H\|_2^{2}\sqrt[16]{\varepsilon}\|\varphi P\|^2_2
\end{align}
Now, by Lemma \ref{estimplus}, we have
\begin{align}\label{esti4}\normdfonc{(|\H|-\normdfonc{\H})\varphi P}^2&\leqslant\frac{m_k}{v_M}\normdpoly{P}^2\insm||H|-\normdfonc{H}|^2|X|^{2k}\varphi^2\vol\notag\\
&\leqslant\frac{C(n,k)}{\normdfonc{H}^{2k}}\normdpoly{P}^2\normdfonc{\varphi(|H|-\normdfonc{H})}^2\notag\\
&\leqslant C(n,k)\mu_k\normdfonc{H}^2\sqrt[16]{\varepsilon}\normdfonc{\varphi P}^2
\end{align}
By the same way, we get
\begin{align}\label{esti8}
\|\varphi(|\H|-\|\H\|_2)dP(\nu)\|_2^2\leqslant C(n,k)\mu_k\sqrt[16]{\varepsilon}\|\H\|_2^4\|\varphi P\|_2^2
\end{align}
Now let us estimate the last terms of \eqref{esti1}
\begin{align}\label{esti5}\normdfonc{\varphi|\nabla^0 dP||Z|}^2&\leqslant \frac{C(n,k)\mu_k}{v_M}\normdpoly{P}^2\int_M\varphi^2|X|^{2k-4}|Z|^2\vol\notag\\
&\leqslant C(n,k)\mu_k\normdfonc{H}^{4}\sqrt[16]{\varepsilon}\|\varphi P\|_2^2
\end{align}

Reporting  \eqref{esti2}, \eqref{esti3}, \eqref{esti7}, \eqref{esti4}, \eqref{esti8} and \eqref{esti5} in \eqref{esti1} we get
$$\normdfonc{\Delta(\varphi P)-\mu_k\normdfonc{H}^2\varphi P}\leqslant C(n,k)\sqrt[16]{\varepsilon} \mu_k\normdfonc{H}^2\normdfonc{\varphi P}$$
\end{proof}

Let $E_k^\varepsilon$ be the space spanned by the eigenfunctions of $M$ associated to an eigenvalue in the interval $\bigl[(1-\sqrt[16]{\varepsilon}2C(n,k))\mu_k^{S_M},(1+\sqrt[16]{\varepsilon}2C(k,n))\mu_k^{S_M}\bigr]$. If $\dim E_k^\varepsilon<m_k$, then there exists $\varphi P\in(\varphi \hkm)\setminus\{0\}$ which is $L^2$-orthogonal to $E_k^{\nu}$. Let $\displaystyle\varphi P=\sum_i f_i$ be the decomposition of $\varphi P$ in the Hilbert basis given by the eigenfunctions $f_i$ of $M$ associated respectively to $\lambda_i$. Putting $N:=\{i/\, f_i\notin E_k^{\varepsilon}\}$, by assumption on $P$ we have 
$$\displaylines{4C(n,k)^2\sqrt[8]{\varepsilon}(\mu_k^{S_M})^2\|\varphi P\|_2^2\leqslant\sum_{i\in N}\bigl(\lambda_i-\mu_k^{S_M}\bigr)^2\|f_i\|_2^2=\|\Delta (\varphi P)-\mu_k^{S_M}\varphi P\|_2^2\hfill\cr
\leqslant (\mu_k^{S_M})^2C(n,k)^2\sqrt[8]{\varepsilon}\|\varphi P\|_2^2}$$
which gives a contradiction. We then have $\dim E_k^\varepsilon\geqslant m_k$.


\section{Some examples}\label{se}

\subsection{Proof of Theorem \ref{ctrexple4}}\label{pot2}

We adapt the constructions made in \cite{Ann,Tak1,AG1}. We consider submanifolds obtained by connected sum of a small submanifold $\varepsilon M_2$ with a fixed submanifold $M_1$ along a small, adequately pinched cylinder $\varepsilon T'_\varepsilon$. This is a 2 scales collapsing sequence of submanifolds. Gluing several such cylinders (with $M_2$ replaced by $\S^m$ for the supplementary cylinders) adds any finite set of eigenvalues to the spectrum of $M_1$. Since $F\setminus{\rm Sp}(M_1)$ is the Hausdorff limit of a sequence of finite sets, this will give Theorem \ref{ctrexple4}. We first describe precisely the construction in the case of one gluing, taking a special care of the case $\alpha=n$.

\subsubsection{Flattening of submanifolds}
For any submanifold $M$ of $\R^{n+1}$, we set $\tilde{M}^\varepsilon$ the submanifold of $\R^{n+1}$ obtained by flattening $M$ at the neighbourhood of a point $x_0\in M$ along the following procedure: 

$M$ is locally equal to $\{x_0+w+f(w),\,w\in B_0(10\varepsilon_0)\subset T_{x_0}M\}$ where $f:B_0(10\varepsilon_0)\subset T_{x_0}M\to N_{x_0}M$ is a smooth function and $N_{x_0}M$ is the normal bundle $M$ at $x_0$. 
Let $\varphi:\R_+\to[0,1]$ be a smooth function such that $\varphi=0$ on $[0,4\varepsilon_0]$ and $\varphi=1$ on $[5\varepsilon_0,+\infty)$. 
We set $\tilde{M}^\varepsilon$ the submanifold obtained by replacing the subset $\{x_0+w+f(w),\,w\in B_0(10\varepsilon_0)\subset T_{x_0}M\}$ by $\{x_0+w+f_\varepsilon(w),\,w\in B_0(10\varepsilon_0)\subset T_{x_0}M\}$, with $f_\varepsilon(w)=f\bigl(\varphi(\frac{\varepsilon_0\|w\|}{\varepsilon})w\bigr)$ for any $\varepsilon\leqslant2\varepsilon_0$, and $M^\varepsilon=\tilde{M}^\varepsilon\setminus B_{x_0}(3\varepsilon)$. 
Note that $\tilde{M}^\varepsilon$ is a smooth deformation of $M$ in a neighbourhood of $x_0$ and that the boundary of $M^\varepsilon$ has a neighbourhood isometric to the flat annulus $B_0(4\varepsilon)\setminus B_0(3\varepsilon)$ in $\R^m$. 
Note also for what follows that for $\varepsilon$ small enough, $M^\varepsilon\setminus B_{x_0}(10\varepsilon)$ is a subset of $M$. 
As a graph, the curvatures of $\tilde{M}^\varepsilon$ at the neighbourhood of $x_0$ are given by the formulae
\begin{align*}
|\B_\varepsilon|^2&=\sum_{i,j,k,l=1}^m\sum_{p,q=m+1}^{n+1}Ddf_p(e_i,e_k)Ddf_q(e_j,e_l)H^{i,j}H^{k,l}G^{p,q}\\
\H_\varepsilon&=\frac{1}{m}\sum_{k,l=m+1}^{n+1}\sum_{i,j=1}^mDdf_k(e_i,e_j)H^{i,j}G^{k,l}(\nabla f_l-e_l)
\end{align*}
where $(e_1,\cdots,e_m)$ is an ONB of $T_{x_0}M_1$, $(e_{m+1},\cdots,e_{n+1})$ an ONB of $N_{x_0}M_1$, $f_\varepsilon(w)=\sum_{i=m+1}^{n+1}f_i(w)e_i$, $G_{kl}=\delta_{kl}+\langle\nabla f_k,\nabla f_l\rangle$ and $H_{kl}=\delta_{kl}+\langle df_\varepsilon(e_k),df_\varepsilon(e_l)\rangle$.
Now $f_\varepsilon$ converges in $\mathcal{C}^\infty$ norm to $f$ on any compact subset of $B_0(\varepsilon_0)\setminus\{0\}$, while $|df_\varepsilon|$ and $|Ddf_\varepsilon|$ remain uniformly bounded on $B_0(\varepsilon_0)$ when $\varepsilon$ tends to $0$. 
By the Lebesgue convergence theorem, we have
\begin{align*}
\lim_{\varepsilon\to 0}\int_{\tilde{M}^\varepsilon}|\H_\varepsilon|^{\alpha}dv=\lim_{\varepsilon\to 0}\int_{M^\varepsilon}|\H_\varepsilon|^{\alpha}dv= \int_{M}|\H|^{\alpha}dv\\
\lim_{\varepsilon\to 0}\int_{\tilde{M}^\varepsilon}|\B_\varepsilon|^{\alpha}dv=\lim_{\varepsilon\to 0}\int_{M^\varepsilon}|\B_\varepsilon|^{\alpha}dv= \int_{M}|\B|^{\alpha}dv
\end{align*}
for any $\alpha\geqslant 1$.
By the same way, any function on $M$ can be seen as a function on $\tilde{M}^\varepsilon$ and this  identification of $H^1(M)$ with $H^1(\tilde{M}^\varepsilon)$ tends to an isometry as $\varepsilon$ tends to $0$.

\subsubsection{Control of the curvature of the gluing}
Let $M_1$, $M_2$ be 2 manifolds of dimension $m$ isometrically immersed in $\R^{n+1}$ and $\lambda,L$ be some fixed, positive real numbers, with $\lambda\notin{\rm Sp}(M_1)$ and $L>\max\bigl(\frac{C(M_1)(1+\lambda)^2}{d^2},1\bigr)$, where $d$ is the distance between $\lambda$ and ${\rm Sp}(M_1)$ in $\R$. We consider the flattenings $\tilde{M}_2^\varepsilon$ of $M_2$ around the point $x_2$ and $M_1^\varepsilon$ of $M_1$ around $x_1$. Let $D$ be a smooth hypersurface of revolution of $\R^{m+1}$, composed of three parts, $D_1$, $D_2$, $D_3$, where $D_1$ is a cylinder of revolution isometric to $B_0(3)\setminus B_0(2)\subset\R^{m+1}$ at the neighbourhood of one of its boundary component and isometric to $[0,1]\times\S^{m-1}$ at the neighbourhood of its other boundary component, where $D_2=[0,L]\times\S^{m-1}$ and where $D_3$ is a disc of revolution with pole $x_3$ and isometric to $[0,1]\times\S^{m-1}$ at its boundary and to a flat disc at the neighbourhood of $x_3$. Let $C$ be a cylinder of revolution of dimension $m$ isometric to $B_0(2)\setminus B_0(1)\subset\R^m$ at the neighbourhood of its 2 boundary components. 
\begin{center}
\includegraphics[width=4cm]{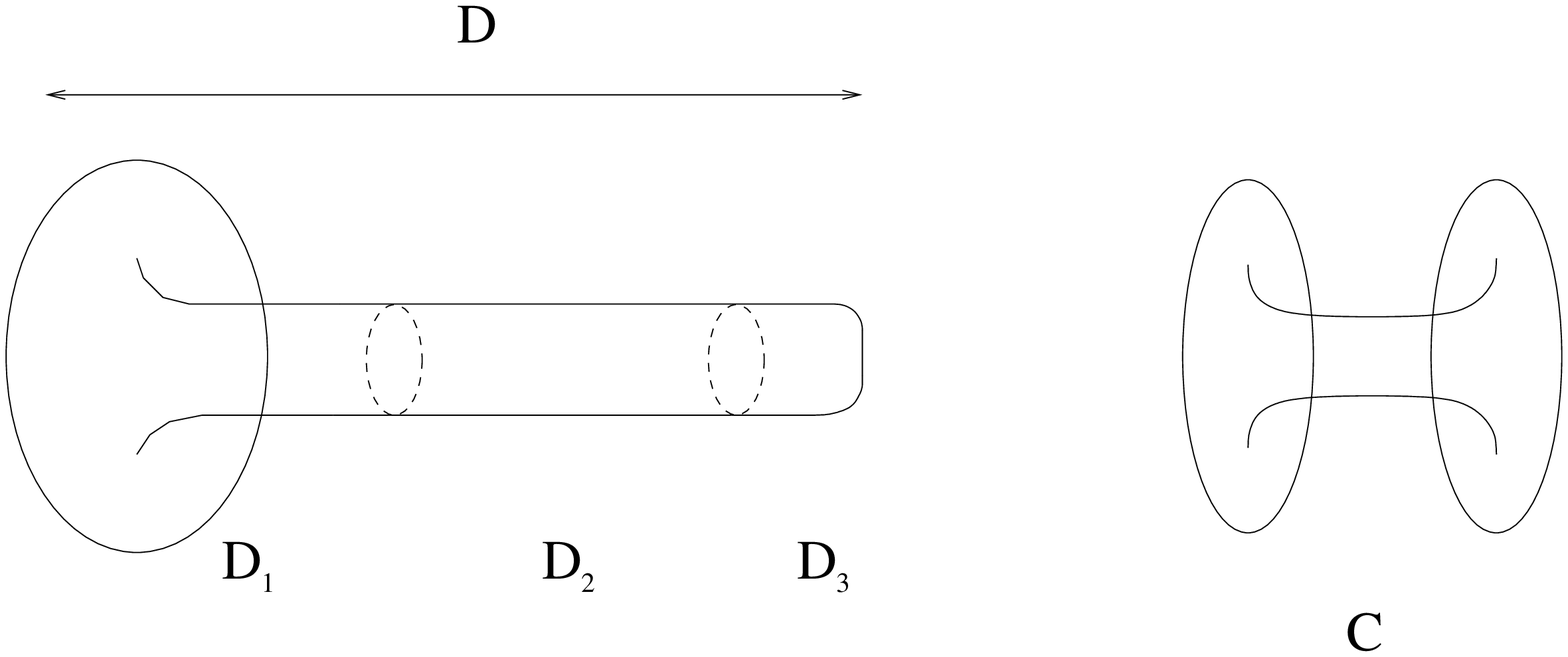}
\end{center}
There exists $\nu_0>0$ such that for any $\nu\in]0,\nu_0[$ the gluing of $\tilde{M}^\varepsilon_2\setminus B_{x_2}(2\nu)$, of $\nu C$ and of $D\setminus B_{x_3}(2\nu)$ along their isometric boundary components is a smoothly immersed submanifold $T'_\nu$ of dimension $m$.
 By standard arguments (see  for instance \cite{Ann} or what is done in section \ref{computspec} in a more complicate case), when $\nu$ tends to $0$, the Dirichlet spectrum of $T'_\nu$ converges to the disjoint union of the Dirichlet spectrum of $D$ and of the spectrum of $M_2$. Moreover, for $\nu$ small enough, $\lambda_1^D(T'_\nu)$ depends continuously on $\nu$. We infer that for any $\varepsilon\in]0,\varepsilon_0(M_2,\lambda,L,D_1,D_3)[$ there exists a $\nu_\varepsilon\in]0,\nu_0(M_2,\lambda,L,D_1,D_3)[$ such that $\lambda_1^D(T'_{\nu_\varepsilon})=\varepsilon^2\lambda$ and $\lambda_2^D(T'_{\nu_\varepsilon})\geqslant\Lambda_2(L,M_2,\lambda,D_1,D_3)>0$. We set $T_\varepsilon=\varepsilon T'_{\nu_\varepsilon}$. Note that we have $\int_{T_\varepsilon}|\B|^p\leqslant \varepsilon^{m-p}C_2(M_2,\lambda,L,D_1,D_3)$ for any $p<m$, $\lim_{\varepsilon\to 0}\int_{T_\varepsilon}|\B|^m=\int_{M_2}|\B|^m+\int_{D_1}|\B|^m+\int_{D_3}|\B|^m+LC(m)$, $\lambda^D_1(T_\varepsilon)=\lambda$ and $\lambda_2^D(T_\varepsilon)\geqslant\frac{\Lambda_2}{\varepsilon^2}$ for any $\varepsilon\leqslant\varepsilon_0$.
\begin{center}
\includegraphics[width=4cm]{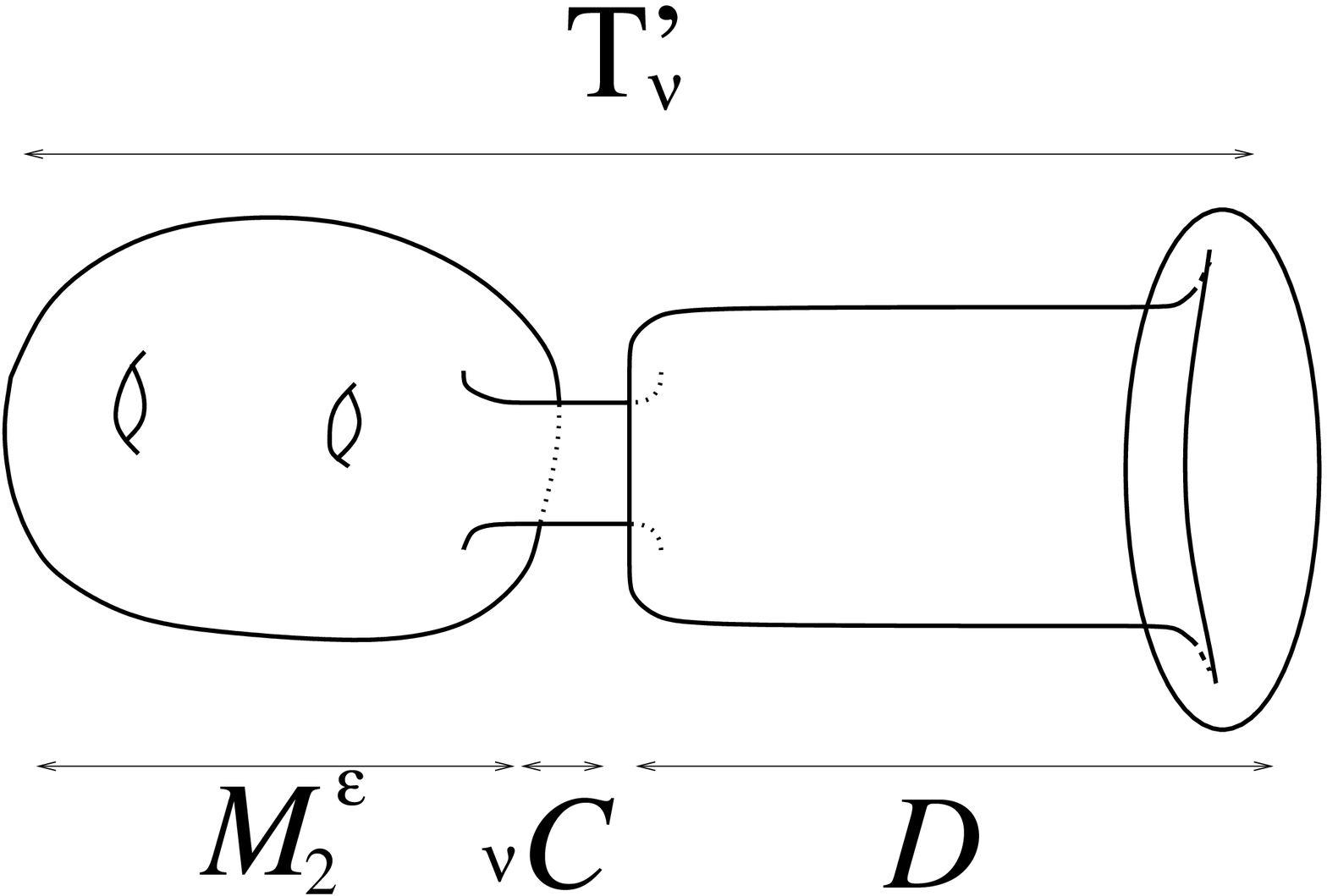}
\end{center}
We set $M_\varepsilon$ the $m$-submanifold of $\R^{n+1}$ obtained by gluing $M_1^\varepsilon$ and $T_{\varepsilon}$ along their boundaries in a fixed direction $\nu\in N_{x_1}M_1$. Note that $M_\varepsilon$ is a smooth immersion of $M_1\#M_2$  (resp. an embedding when $M_1$ and $M_2$ are embedded).
\begin{center}
\includegraphics[width=4cm]{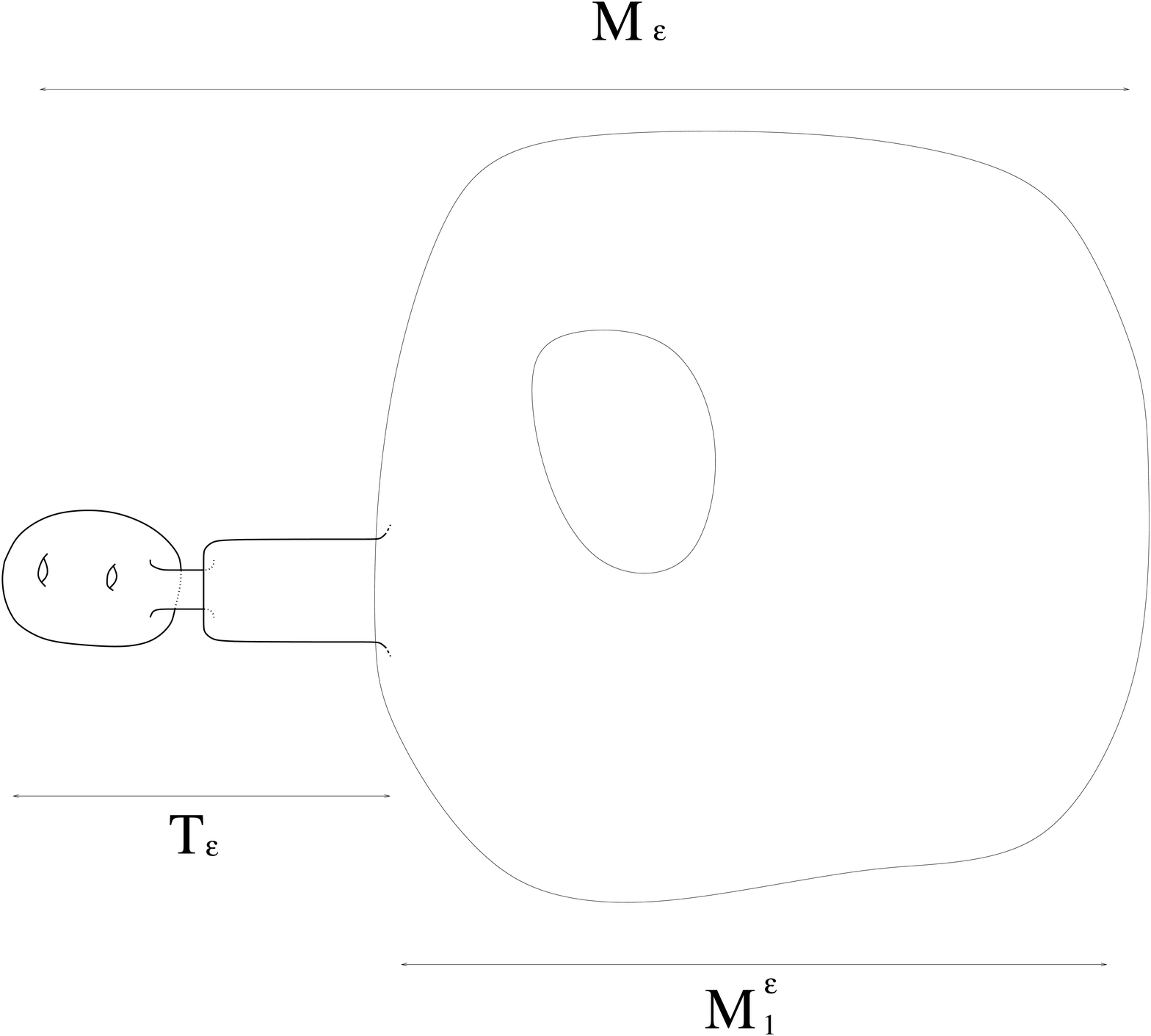}
\end{center}

By the computations above, the sequence $i_k(M_1\#M_2)=M_\frac{1}{k}$ converges to $M_1$ in Hausdorff distance and we have
\begin{align*}
\lim_{\varepsilon\to 0}\int_{M_\varepsilon}|\H_\varepsilon|^{\alpha}dv= \int_{M_1}|\H|^{\alpha}dv\quad\quad\quad \lim_{\varepsilon\to 0}\int_{M_\varepsilon}|\B_\varepsilon|^{\alpha}dv=\int_{M_1}|\B|^{\alpha}dv
\end{align*}
for any $\alpha<m$ and
\begin{align*}
\lim_{\varepsilon\to 0}\int_{M_\varepsilon}|\H_\varepsilon|^m= \int_{M_1}|\H|^{m}+\int_{D_1\cup D_3}|\H|^{m}+C(m)L+\int_{M_2}|\H|^{m}\\
\lim_{\varepsilon\to 0}\int_{M_\varepsilon}|\B_\varepsilon|^{m}=\int_{M_1}|\B|^{m}+\int_{D_1\cup D_3}|\B|^{m}+C(m)L+\int_{M_2}|\B|^{m}
\end{align*}

\subsubsection{Computation of the spectrum of $M_\varepsilon$}\label{computspec}
We will prove that there exists a sequence $(\varepsilon_p)_{p\in\N}$ such that $\varepsilon_p\to0$ and the spectrum of $M_{\varepsilon_p}$ converges to the disjoint union of ${\rm Sp}(M_1)$ and of $\{\tilde{\lambda}\}$, where $\tilde{\lambda}$ satisfies $\lambda-\frac{C(M_1)(1+\lambda)}{\sqrt{L}}\leqslant\tilde{\lambda}\leqslant\lambda$. 
Since the collapsing of $M_\varepsilon$ is multiscale, the  cutting and rescaling technique of \cite{Ann, Tak1} has to be adapted. Indeed, after rescaling of $T_\varepsilon$ we get another collapsing sequence of submanifolds with no uniform control of the trace and Sobolev Inequalities.

We denote by $(\lambda_k)_{k\in\N}$ the union with multiplicities of the spectrum of $M_1$ and of $\{\lambda\}$, by $(\lambda_k^\varepsilon)_{k\in\N}$ the spectrum of $M_\varepsilon$ and by $(\mu_k^\varepsilon)_{k\in\N}$ the Dirichlet spectrum of the disjoint union $M'_\varepsilon= T_\varepsilon\cup \Bigl(M_1^\varepsilon\setminus B_{x_1}(10\varepsilon)\Bigr)$.
By the Dirichlet principle, we have $\lambda_k^\varepsilon\leqslant\mu_k^\varepsilon$ for any $k\in\N$. 
It is well known (see for instance \cite{Co}) that the Dirichlet spectrum of $M_1^\varepsilon\setminus B_{x_1}(10\varepsilon)$ converges to the spectrum of $M_1$. 
We infer that $\mu_k^\varepsilon\to\lambda_k$ as $ \varepsilon\to 0$  and so $\limsup\lambda_k^\varepsilon\leqslant\lambda_k$ for any $k\in\N$. 

We set $\alpha_k=\liminf_{\varepsilon\to 0}\lambda_k^\varepsilon$.
To get some lower bound on the $\alpha_k$, we need some local trace inequalities. 
We set $S_t=\{x\in T_\varepsilon/\,d(x,\partial T_\varepsilon)=-t\}$ for any $t\leqslant0$ and $S_t=\{x\in M_1^\varepsilon/\,d(x,\partial M_1^\varepsilon)=t\}$ for any $t\geqslant0$. 
We also set $B_{t,r}=\cup_{\{s/\,|s-t|\leqslant r\}}S_s$, $N_r=M_1^\varepsilon\cup B_{\frac{r}{2},\frac{-r}{2}}$ for any $r\leqslant0$ and $N_r=M_1^\varepsilon\setminus B_{\frac{r}{2},\frac{r}{2}}$ for any $r\geqslant0$.
Let $a_{M_1}$ be a constant such that the volume density $\theta_\varepsilon$ of $M_\varepsilon$ in normal coordinates to $ S_{-2\varepsilon}$ satisfies $\frac{1}{a_{M_1}}(3+\frac{t}{\varepsilon})^{m-1}\geqslant\theta_\varepsilon(t,u)\geqslant a_{M_1}(3+\frac{t}{\varepsilon})^{m-1}$ for any $t\in[-2\varepsilon,a_{M_1}]$ and any $u\in S_{-2\varepsilon}$. 
Let $\varepsilon d$ be the distance in $M_\varepsilon$ between $M_1^\varepsilon$ and $\varepsilon D_2$ and $C(D_1)$ be a constant such that for any $t\in[-(L+d+2)\varepsilon,-2\varepsilon]$ and any $u\in S_{-2\varepsilon}$ we have $\frac{\theta_\varepsilon(t,u)}{\theta_\varepsilon(-2\varepsilon,u)}\in[\frac{1}{C(D_1)},C(D_1)]$.
Let $\eta:[-2\varepsilon,a_{M_1}]\to[0,1]$ be a smooth function such that $\eta(t)=1$ for any $t\leqslant\frac{a_{M_1}}{2}$, $\eta(a_{M_1})=0$ and $|\eta'|\leqslant\frac{4}{a_{M_1}}$. 
For any $r\in[-2\varepsilon,a_{M_1}/2]$ and any $f\in H^1(M_\varepsilon)$, we have
\begin{align*}
\int_{S_{r}}f^2&=\int_{S_{-2\varepsilon}}\Bigl(\int_{r}^{a_{M_1}}\frac{\partial}{\partial s}[\eta(\cdot)f(\cdot,u)]d s\Bigr)^2\theta_\varepsilon(r,u) du\\
\leqslant&\int_{r}^{a_{M_1}}\frac{\sup_{u\in S_{-2\varepsilon}}\theta_\varepsilon(r,u)}{\inf_{u\in S_{-2\varepsilon}}\theta_\varepsilon(s,u)}ds\int_{S_{-2\varepsilon}}\int_{r}^{a_{M_1}}\bigl(\frac{\partial}{\partial s} [\eta(\cdot)f(\cdot,\frac{\varepsilon}{10}u)]\bigr)^2 \theta_\varepsilon(s,u)\\
\leqslant& c(M_1)\int_{r}^{a_{M_1}}\frac{(3+r/\varepsilon)^{m-1}}{(3+s/\varepsilon)^{m-1}}ds\|f\|_{H^1(M_\varepsilon)}^2
\end{align*}
which gives
\begin{equation}\label{Speneg}
\int_{S_r}f^2\leqslant c(M_1)(3\varepsilon+r)\|f\|_{H^1(M_\varepsilon)}^2
\end{equation}
when $m\geqslant 3$. By the same way, for any $r\in[-(L+d+2)\varepsilon,-2\varepsilon]$, we have
\begin{equation}\label{Spepos}
\int_{S_{r}}f^2\leqslant -c(M_1,D_1)(\varepsilon+r)\|f\|_{H^1(M_\varepsilon)}^2
\end{equation}

We now use this local trace inequality to get some estimates on the eigenfunctions of $M_\varepsilon$. We set $\varphi:M_\varepsilon\to[0,1]$ be a smooth function equal to $1$ on $N_{21\varepsilon/2}\cup \bigl(M_\varepsilon\setminus N_{-\varepsilon/2}\bigr)$, equal to $0$ outside $M'_\varepsilon$ and such that $|\varphi'|\leqslant\frac{4}{\varepsilon}$. For any $f_1,f_2\in H^1(M_\varepsilon)$, integration of Inequalities \eqref{Speneg} and \eqref{Spepos} gives us 
\begin{align}\label{cutscal}
\bigl|\int_{M_\varepsilon}f_1f_2-\int_{M_{\varepsilon}}\varphi f_1\varphi f_2\bigr|&\leqslant \int_{M_{\varepsilon}}|\varphi^2-1| |f_1||f_2|\leqslant c(M_1)\varepsilon^2\|f_1\|_{H^1(M_\varepsilon)}\|f_2\|_{H^1(M_\varepsilon)}
\end{align}
and
\begin{align}
\int_{M_\varepsilon}|d\varphi f_1|^2&\leqslant \int_{M_\varepsilon}|d\varphi|^2f_1^2+2\varphi f_1(df_1,d\varphi)+\varphi^2|df_1|^2\nonumber\\
&\leqslant\frac{16}{\varepsilon^2}\|f_1\|_{L^2({\rm Supp}(d\varphi))}^2+\frac{8}{\varepsilon}\|f_1\|_{L^2({\rm Supp}(d\varphi))}\|df_1\|_2+\|df_1\|_2^2\nonumber\\\label{cutdir}
&\leqslant c(M_1)\|f_1\|^2_{H^1(M_\varepsilon)}
\end{align}

Let $(f_k^\varepsilon)$ be a $L^2$-orthonormal, complete set of eigenfunctions of $M_\varepsilon$.
For any $k$, we set $\tilde{f}_k^\varepsilon$ the function on $M_1$ equal to $\varphi f_k^\varepsilon$ on $N_{10\varepsilon}$ and extended by $0$. By Inequality \eqref{cutdir}, we have $\|\tilde{f}_k^\varepsilon\|_{H^1( M_1)}^2\leqslant c(M_1)(1+\lambda_k)$ for $\varepsilon$ small enough. We infer by diagonal extraction that there exists some sequences $(\varepsilon_p)_{p\in\N}$ and $(h_k)_{k\in\N}\in H^1(M_1)^\N$ such that  $\lambda_k^{\varepsilon_p}\to\alpha_k$ and $(\tilde{f}_k^{\varepsilon_p})_p$ converges weakly in $H^1(M_1)$ and strongly  in $L^2(M_1)$ to $h_k$, for any $k$. It is easy to prove that $h_k$ is a weak solution of $\Delta h_k=\alpha_k h_k$ on $H^1(M_1\setminus\{x_1\})=H^1(M_1)$. By elliptic regularity, either $h_k=0$ or $\alpha_k$ is an eigenvalue of $M_1$.

Let $k_0\in\N$ such that $\lambda_{k_0}=\lambda$. Since $D_2$ isometric to $[0,L]\times\S^{m-1}$, any $f^{\varepsilon_p}_{k}$ can be seen as a function on $[0,\varepsilon_p L]\times\varepsilon_p\S^{m-1}$. For any $f=\sum_{i\leqslant k_0}\beta_if_i^{\varepsilon_p}\in{\rm Vect}\{f_i^{\varepsilon_p}/\,i\leqslant k_0\}$, we define the rescaling $F_{p}$ on $c=[0,1]\times\S^{m-1}$ by  $F_{p}(t,x)=\varepsilon_p^{\frac{m}{2}-1}L^{-\frac{1}{2}}f(\varepsilon_p L t,\varepsilon_p x)$. By Inequality \eqref{Spepos}, we have 
$$\displaylines{
\int_cF^2_{p}=\frac{1}{\varepsilon_p^2L^2}\int_{\varepsilon_p D_2}f^2\leqslant c(M_1,D_1)(1+\frac{2}{L})(1+\lambda)\|f\|_2^2,\cr \int_{\{0\}\times\S^{m-1}}F_{p}^2=\frac{1}{L\varepsilon_p}\int_{\varepsilon_p( D_1\cap D_2)}f^2\leqslant \frac{c(M_1,D_1)(1+\lambda)\|f\|_2^2}{L},\cr 
\mbox{and }\ \ \ \int_{\{1\}\times\S^{m-1}}F_{p}^2=\frac{1}{L\varepsilon_p}\int_{\varepsilon_p( D_3\cap D_2)}f^2\leqslant (1+\lambda)c(M_1,D_1)(1+\frac{d}{L})\|f\|_2^2,
}$$
for $p$ large enough (note that we have $d\geqslant 2$ by construction). Moreover, we have $\int_c|dF_{p}|^2\leqslant\int_{\varepsilon_pD_2}|df|^2\leqslant \lambda\|f\|_2^2$. So we can assume that there exists $F_{\infty}\in H^1(c)$ such that the sequence $(F_{p})$ converges to $F_{\infty}$ weakly in $H^1(c)$ and strongly in $L^2(c)$. 
We set $j_{p}(t)=\int_{\S^{m-1}}F_{p}(t,x)dx$ and $j_{\infty}(t)=\int_{\S^{m-1}}F_{\infty}(t,x)dx$, we have $j_{p},j_{\infty}\in H^1([0,1])$ (with $j_{p}'(t)=\int_{\S^{m-1}}\frac{\partial F_{p}}{\partial t}(t,x)dx$), $j_{p}\to j_{\infty}$ strongly in $L^2([0,1])$ and weakly in $H^1([0,1])$. By the estimates above and the compactness of the trace operator on $c$, we have $|j_{\infty}(0)|\leqslant\frac{C(M_1)\sqrt{1+\lambda}\|f\|_2}{\sqrt{L}}$ and $|j_{\infty}(1)|\leqslant \sqrt{1+\lambda}\|f\|_2C(M_1)$. Hence $l(t)=j_{\infty}(t)-\bigl(j_{\infty}(0)+(j_{\infty}(1)-j_{\infty}(0))t\bigr)$ is in $H^1_0([0,1])$. For any $\psi\in\mathcal{C}_c^\infty([0,1])$, we set $\psi_p(t,x)=\varepsilon_p L\psi(\frac{t}{\varepsilon_p L})$ seen as a function in $H^1_0(\varepsilon_p D_2)$. We have
\begin{align*}&\int_0^1l'\psi'\,dt=\int_0^1j_{\infty}'\psi'\,dt\\
&=\lim_p\int_0^1j_{p}'(t)\psi'(t)\,dt=\lim_p\int_{c}\frac{\partial F_{p}}{\partial t}\psi'=\lim_p\frac{1}{\varepsilon_p^\frac{m}{2}\sqrt{L}}\int_{\varepsilon_p D_2}\langle df,d\psi_p\rangle\, dt\, dx\\
&=\lim_p\sum_i\frac{\beta_i\lambda_i^{\varepsilon_p}}{\varepsilon_p^\frac{m}{2}\sqrt{L}}\int_{\varepsilon_p D_2}f_i^{\varepsilon_p}\psi_p\, dt\, dx=\sum_i\alpha_i\beta_i L^2\lim_p\varepsilon_p^2\int_{c}F_{i,p}\psi\, dt\, dx\\
&=0,
\end{align*}
where $F_{i,p}(t,x)=\varepsilon_p^{\frac{m}{2}-1}L^{-\frac{1}{2}}f_i^{\varepsilon_p}(\varepsilon_p L t,\varepsilon_p x)$. We infer $l$ is harmonic and in $H^1_0([0,1])$, i.e. $l=0$ and $j_{\infty}(t)=j_{\infty}(0)+(j_{\infty}(1)-j_{\infty}(0))t$ on $[0,1]$. Since the Poincare inequality on $\S^{m-1}$ gives us 
\begin{align*}
\int_{\S^{m-1}}F_p(t,x)^2\,dx&\leqslant\frac{1}{\Vol\S^{m-1}}\bigl(\int_{\S^{m-1}}F_p(t,x)\,dx\bigr)^2+\frac{1}{m-1}\int_{\S^{m-1}}|d_{\S^{m-1}}F_p|^2\\
&\leqslant\frac{1}{\Vol\S^{m-1}}j_p^2(t)+\frac{\varepsilon_p}{(m-1)L}\int_{\varepsilon_p\S^{m-1}}|d_{\varepsilon_p\S^{m-1}}f|^2(\varepsilon_p L t,x)\,dx,
\end{align*}
 we get that
\begin{align*}
\frac{1}{L\varepsilon_p^2}&\int_{[0,\varepsilon_p\sqrt{L}]\times\varepsilon_p\S^{m-1}}f^2=L\int_{[0,\frac{1}{\sqrt{L}}]\times \S^{m-1}}F_{p}^2\\
&\leqslant\frac{L}{\Vol\S^{m-1}}\int_0^{\frac{1}{\sqrt{L}}} j_p^2(t)\,dt+\frac{1}{(m-1)L}\int_{[0,\varepsilon_p\sqrt{L}]\times\varepsilon_p\S^{m-1}}|d_{\varepsilon_p\S^{m-1}}f|^2\\
&\leqslant\frac{L}{\Vol\S^{m-1}}\int_0^{\frac{1}{\sqrt{L}}} j_p^2(t)\,dt+\frac{\lambda}{(m-1)L}\|f\|^2\\
&\to \frac{L}{\Vol\S^{m-1}}\int_0^{\frac{1}{\sqrt{L}}} j_{\infty}(t)^2\, dt+\frac{\lambda}{(m-1)L}\|f\|^2\leqslant \frac{C(M_1)(1+\lambda)\|f\|^2}{\sqrt{L}}
\end{align*}

If the family $(h_i)_{i< k_0}$ is not free in $L^2(M_1)$, then either one $h_i$ is null or they are all eigenfunctions of $M_1$. Since the eigenspaces are in direct sum, we infer that there exists $\mu\leqslant\lambda_{k_0-1}$ and $(\beta_i)\in\R^{k_0}\setminus\{0\}$ such that $\sum_i\beta_i^2=1$, $\sum_i\beta_ih_i=0$ and $\alpha_i=\mu$ for any $i$ such that $\beta_i\neq 0$. We set $f=\sum_i\beta_if_i^{\varepsilon_p}$ and $\eta:M_{\varepsilon_p}\to[0,1]$ a smooth function equal to $1$ on $M_{\varepsilon_p}\setminus N_{-(2+d+\sqrt{L})\varepsilon_p}$, equal to $0$ on $N_{-(2+d)\varepsilon_p}$ and such that $|\varphi'|\leqslant\frac{2}{\varepsilon_p\sqrt{L}}$.  We then have
\begin{align}
&|\int_{M_{\varepsilon_p}}|d(\eta f)|^2-\mu\int_{M_{\varepsilon_p}}(\eta f)^2\bigr|=\bigl|\int_{M_{\varepsilon_p}}|d\eta|^2f^2+\langle df,d(\eta^2f)\rangle -\mu\int_{M_{\varepsilon_p}}(\eta f)^2\bigr|\nonumber\\
\label{strict}& \leqslant \frac{C(M_1)(1+\lambda)}{\sqrt{L}}\|f\|^2_2+\int_{M_{\varepsilon_p}}\sum_{i,j}(\lambda_i^{\varepsilon_p}-\mu)\beta_i\eta f_i^{\varepsilon_p}\beta_j\eta f_j^{\varepsilon_p}
\end{align}
Inequalities \eqref{cutscal}  and \eqref{Spepos} imply that $\int_{M_{\varepsilon_p}}(\eta f)^2\to 1$. Since $\eta f\in H^1_0(T_{\varepsilon_p})$ and since by construction of $T_{\varepsilon_p}$, we have $\lambda_1^D(T_{\varepsilon_p})=\lambda$, we then have $\int_{M_{\varepsilon_p}}|d(\eta f)|^2\geqslant \lambda\int_{M_{\varepsilon_p}}(\eta f)^2$. Letting $p$ tend to $\infty$ in Inequality  \eqref{strict} we get that $\lambda-\lambda_{k_0-1}\leqslant \frac{C(M_1)(1+\lambda)}{\sqrt{L}}$, which contradicts the choice made on $L$ at the beginning of this subsection. 

We infer that $(h_i)_{i<k_0}$ is free in $L^2(M_1)$. This implies that $\alpha_i$ is an eigenvalue of $M_1$ and $h_i$ is an eigenfunction of $M_1$ for any $i<k_0$. Since $\alpha_i=\lim \lambda_i^{\varepsilon_p}\leqslant\lambda_i=\lambda_i(M_1)$ for any $i<k_0$, we infer that $\alpha_i=\lambda_i$ for any $i<k_0$ and that the $(h_i)_{i<k_0}$ is a basis of the eigenspaces of $M_1$ associated to the first $k_0$ eigenvalues.
By the same way, if $h_{k_0}\neq 0$, then $\alpha_{k_0}=\lambda_{k_0-1}$ (since it is an eigenvalue of $M_1$ less than $\lambda$) and so the family $(h_i)_{i\leqslant k_0}$ is not free. The same argument as above gives a contradiction. So we have that $h_{k_0}=0$.

Assume that there exists another index $l\neq k_0$ such that $h_l=0$. Then, Inequality \eqref{cutscal} gives that $\int_{T_{\varepsilon_p}}\varphi f_{k_0}^{\varepsilon_p}\varphi f_l^{\varepsilon_p}\to 0$, $\int_{T_{\varepsilon_p}}(\varphi f_{k_0}^{\varepsilon_p})^2\to 1$ and $\int_{T_{\varepsilon_p}}(\varphi f_l^{\varepsilon_p})^2\to 1$ and Inequality \eqref{cutdir} gives that $\int_{T_{\varepsilon_p}}|d\varphi f_{k_0}^{\varepsilon_p}|^2$ and $\int_{T_{\varepsilon_p}}|d\varphi f_{l}^{\varepsilon_p}|^2$ remain bounded as $\varepsilon_p\to 0$. We set $g_p$ a unitary eigenfunction of $T_{\varepsilon_p}$ for the Dirichlet problem associated to the eigenvalue $\lambda$. If we set $(\varphi f_{k_0}^{\varepsilon_p})_{|T_{\varepsilon_p}}=\beta_{k_0}^pg_p+\delta_{k_0}^p$ and $(\varphi f_l^{\varepsilon_p})_{|T_{\varepsilon_p}}=\beta_l^pg_p+\delta_l^p$, with 
$\beta^p_{k_0},\beta^p_l\in\R$ and $\delta_{k_0}^p,\delta_l^p$ orthogonal to $g_p$ in $H^1_0(T_{\varepsilon_p})$. The previous relations and the lower bound on $\lambda_2^D(T_{\varepsilon_p})$ imply that 
$$\int_{T_{\varepsilon_p}}|d(\varepsilon f_{k_0}^{\varepsilon_p})|^2\geqslant\lambda(\beta_{k_0}^p)^2µ+\lambda_2^D(T_{\varepsilon_p})\|\delta_{k_0}^p\|_{L^2(T_{\varepsilon_p})}^2\geqslant(\beta^p_{k_0})^2\lambda+\frac{\Lambda_2}{\varepsilon_p^2}\|\delta^p_{k_0}\|^2_{L^2(T_{\varepsilon_p})}.$$
By the same way, $(\beta^p_l)^2\lambda+\frac{\Lambda_2}{\varepsilon_p^2}\|\delta^p_l\|^2_{L^2(T_{\varepsilon_p})}$ is bounded, and so $\|\delta^p_{k_0}\|^2_{L^2(T_{\varepsilon_p})}$ and $\|\delta^p_l\|^2_{L^2(T_{\varepsilon_p})}$ tend to $0$ with $\varepsilon_p$. Now, we have $(\beta_{k_0}^p)^2+\|\delta^p_{k_0}\|^2_{L^2(T_{\varepsilon_p})}\to1$ and so $|\beta_{k_0}^p|\to 1$. By the same way, we have $|\beta_l^p|\to 1$, which contradicts the fact that $\int_{T_{\varepsilon_p}}\varphi f_{k_0}^{\varepsilon_p}\varphi f_l^{\varepsilon_p}\to 0$. We infer that for any $k\in\N\setminus\{k_0\}$ we have that $\alpha_k$ is an eigenvalue of $M_1$. Moreover, if we decompose $(\varphi f_k^{\varepsilon_p})_{|T_{\varepsilon_p}}=\beta_k^pg_p+\delta_k^p$ as above, Inequality \eqref{cutdir} implies that $(\beta_k^p)^2+\frac{\Lambda_2}{\varepsilon_p^2}\|\delta^p_k\|^2_{L^2(T_{\varepsilon_p})}$ remains bounded and so we have $\lim\|\delta^p_k\|^2_{L^2(T_{\varepsilon_p})}=0$ and Inequality \eqref{cutscal} gives
$$0=\lim\int_{M_\varepsilon} f_{k_0}^{\varepsilon_p}f_{k}^{\varepsilon_p}=\lim \beta_k^p\beta_{k_0}^p=\lim \beta_k^p$$
and so $(\varphi f_{k}^{\varepsilon_p})_{|T_{\varepsilon_p}}\to 0$ in $L^2(T_{\varepsilon_p})$ for any $k\neq k_0$. Once again, Inequality \eqref{cutscal} gives us that for any $k,l\in\N \setminus\{k_0\}$, we have
$$\int_{M_1} h_kh_l=\delta_{kl}.$$
From the min-max principle, it gives that we have $\alpha_k\geqslant\lambda_k$ for any $k\neq k_0$. Since we have $\alpha_k\leqslant \lambda_k$ for any $k\in\N$, we infer that for any $k\in\N\setminus\{k_0\}$ we have $\alpha_k=\lambda_k$. Finally, Inequality \eqref{strict}, applied to $f=f_{k_0}^{\varepsilon_p}$ and $\mu=\alpha_{k_0}$ gives that $\alpha_{k_0}\in[\lambda-\frac{C(M_1)(1+\lambda)}{\sqrt{L}},\lambda]$.

\subsubsection{End of the proof of Theorem \ref{ctrexple4} and case $\alpha=m$}

Since we can take $L$ as large as needed while keeping $\int_{M_\varepsilon}|\B|^\alpha\to\int_{M_1}|B|^\alpha$ for any $\alpha<n$, we get Theorem \ref{ctrexple4} for $F={\rm Sp}(M_1)\cup\{\lambda\}$ by diagonal extraction. Iterating the construction (with $M_2$ replaced by $\S^m$ for any supplementary gluing) we get the result for any disjoint union $F={\rm Sp}(M_1)\cup\{$finite set$\}$ and then for any $F$, since any closed set $F$ is the limit in pointed-Hausdorff topology of a sequence of finite sets.

In the case $\alpha=m$, the limit $\int_{M_\varepsilon}|\B|^m$ depend on $L$ and so we are only able to get a weak version of Theorem \ref{ctrexple4} with $F={\rm Sp}(M_1)\cup G$, where $G$ is a finite set whose elements are known up to an error term and where the point 2) is replaced by $\int_{i_k(M_1\# M_2)}|\B|^m$ is bounded by a constant that depend on $M_1$, $M_2$, $D_1$, $D_3$, $G$ and on the error term.

\subsection{Example \ref{ctrexple1}}\label{pot3}
We set $\ieps=[\varepsilon,\frac{\pi}{2}]$ for $\varepsilon>0$ and let $\varphi : \ieps\longrightarrow (-1,+\infty)$ be a function continuous on $\ieps$ and smooth on $(\varepsilon,\frac{\pi}{2}]$. For any $0\leqslant k\leqslant n-2$, we consider the map
$$\begin{array}{rcl}\Phi_\varphi: \S^{n-k-1}\times\S^{k}\times \ieps &\longrightarrow &\R^{n+1}=\R^{n-k}\oplus\R^{k+1}\\
x=(y,z,r) & \longmapsto &  (1+\varphi(r))(y\sin r+z\cos r)\end{array}$$
whose image $X_\varphi$ is a smooth embedded submanifold (with boundary) diffeomorphic to $\S^n\setminus B(\S^k,\varepsilon)$.
We denote respectively by $\B_q (\varphi)$ and $\H_q(\varphi)$ the second fundamental form and the mean curvature of $X_\varphi$ at the point $q$. They are given by the following formulae.

\begin{lemma}\label{courbmoy} Let $x=(y,z,r)\in\S^{n-k-1}\times\S^{k}\times \ieps $, $q=\Phi_\varphi(x)$ and $(u,v,h)\in T_x \xeps$. Then we have
$$\displaylines{
n\H_{q}(\varphi)=\bigl( \coeff\bigr)^{-3/2}\Bigl[-(1+\varphi(r))\varphi''(r)+(1+\varphi(r))^2+2\varphi'^2(r)\Bigr]\hfill\cr
\hfill+\frac{\bigl( \coeff\bigr)^{-1/2}}{1+\varphi(r)}\Bigl[ -(n-k-1)\varphi'(r)\cot r+(n-1)(1+\varphi(r))+k\varphi'(r)\tan r\Bigr]\cr
\cr
|\B _{q}(\varphi)|=\hfill\cr
\hfill\frac{(1+\varphi(r))^{-1}}{\bigl(1+(\frac{\varphi'(r)}{1+\varphi(r)})^2\bigr)^{1/2}}\max\Bigl(\bigl|1-\frac{\varphi'}{1+\varphi}\cot r\bigr|,\bigl|1+\frac{\varphi'}{1+\varphi}\tan r\bigr|,\bigl|1+\frac{(\varphi')^2-(1+\varphi)\varphi''}{\coeff}\bigr|\Bigl)}$$
\end{lemma}

To prove Theorem \ref{ctrexple1}, we set $a<\frac{\pi}{10}$ and define the function $\fieps$ on $I_\varepsilon$ by
$$\fieps(r)=\left\{\begin{array}{ll}f_{\varepsilon}(r)=\varepsilon\displaystyle\int_1^{\frac{r}{\varepsilon}}\frac{dt}{\sqrt{t^{2(n-k-1)}-1}} & \text{if } \varepsilon\leqslant r\leqslant a+\varepsilon,\\[3mm]
u_{\varepsilon}(r) & \text{if } r\geqslant a+\varepsilon,\\[2mm]
b_\varepsilon& \text{if }r\geqslant 2a+\varepsilon,
\end{array}\right.$$
where $b_\varepsilon$ is a constant and $u_{\varepsilon}$ is chosen so that $\fieps$ is smooth on $(\varepsilon,\frac{\pi}{2}]$ and strictly concave on $(\varepsilon,2a+\varepsilon]$. Since we have $f_{\varepsilon}(x)\to 0$, $f'_{\varepsilon}(x)\to 0$, $f''_{\varepsilon}(x)\to 0$ for any fixed $x\in(\varepsilon,a+\varepsilon]$, the concavity implies that $b_{\varepsilon}\to 0$ as $\varepsilon\to0$ (hence $b_\varepsilon$ can be chosen less than $\frac{1}{2}$), that $\varphi_\varepsilon\to 0$ uniformly on $I_\varepsilon$ and that $\varphi_\varepsilon'$ converges uniformly to $0$ on any compact subset of $(\varepsilon,\frac{\pi}{2}]$. Moreover, $u_\varepsilon$ can be chosen such that $\varphi_\varepsilon''$ converges to $0$ uniformly on any compact subset of  $(\varepsilon,\frac{\pi}{2}]$.

On $(\varepsilon,a+\varepsilon]$, $\varphi_\varepsilon$ satisfies
\begin{align}\label{equadif}\varphi_\varepsilon''=-\frac{(n-k-1)(1+\varphi_\varepsilon'^2)}{r}\varphi_\varepsilon',\end{align} $\fieps(\varepsilon)=0$ and $\displaystyle\lim_{t\to\varepsilon}\varphi_\varepsilon'(t)=+\infty=-\lim_{t\to\varepsilon}\varphi_\varepsilon''(t)$. On $(-b_{\varepsilon},b_{\varepsilon})$, we define $\fiteps$ by $\fiteps(t)=\fieps^{-1}(|t|)$. Since $\fiteps$ satisfies the equation $yy''=(n-k-1)\bigl(1+(y')^2\bigr)$ with initial data $\fiteps(0)=\varepsilon$ and $\fiteps'(0)=0$, it is smooth at $0$, hence on $(-b_\varepsilon,b_\varepsilon)$.

Now we consider the two applications $\Phi_{\varphi_\varepsilon}$ and $\Phi_{-\varphi_\varepsilon}$ defined as above, and we set $\mpeps=X_{\varphi_\varepsilon}$, $\mmeps=X_{-\varphi_\varepsilon}$ and  $\meps^k=\mpeps\cup\mmeps$. $M^k_\varepsilon$ is a smooth submanifold of $\R^{n+1}$ since the function $F_{\varepsilon}(p_1,p_2)= |p_1|^2-|p|^2\sin^2\bigl(\fiteps(|p|-1)\bigr)$, defined on 
$$U=\{p=(p_1,p_2)\in\R^{n-k}\oplus\R^{k+1}/\,p_1\neq 0,\, p_2\neq 0,\, -b_{\varepsilon}+1<|p|<b_{\varepsilon}+1\}$$ gives a smooth, local equation of $M^k_\varepsilon$ at the neighborhood of $\mpeps\cap\mmeps$ that satisfies
$$\nabla F_{\varepsilon}(p_1,p_2)=2p_1\cos^2\varepsilon-2p_2\sin^2\varepsilon\neq 0$$
on  $\mpeps\cap\mmeps$.

We denote respectively by $\Heps$ and $\beps$, the mean curvature and the second fundamental form of $\meps^k$. 

\begin{theorem}\label{ctrexple2} $\|\Heps\|_{\infty}^{~}$ and $\|\B _\varepsilon\|_{n-k}^{~}$ remain bounded whereas $\|\Heps-1\|_1\to 0$ and $\bigl\||X|-1\bigr\|_\infty\to0$ when $\varepsilon\to 0$.
\end{theorem}

\begin{remark}
We have $\|\B _\varepsilon\|_q\to\infty$ when $\varepsilon \to 0$, for any $q>n-k$.
\end{remark}

\begin{proof} From the lemma \ref{courbmoy} and the definition of $\fieps$, $\H_\varepsilon$ and $|\B _\varepsilon|$ converge uniformly to $1$ on any compact of $M_\varepsilon^k\setminus M_\varepsilon^+\cap M_\varepsilon^-$. On the neighborhood of $ M_\varepsilon^+\cap M_\varepsilon^-$, we have $n(\Heps)_x=n\heps(r)$ and $n \heps\leqslant\hepsi+\hepsd+\hepst$, where
$$\displaylines{
\hepsd(r)=k\frac{(\fieps'^2+(1\pm\fieps)^2)^{-1/2}}{1\pm\fieps}\fieps'\tan(r)\leqslant\frac{k}{1-b_\varepsilon}\tan\frac{\pi}{5}\hfill\cr
\hepst(r)=(n-1)(\fieps'^2+(1\pm\fieps)^2)^{-1/2}\hfill\cr
\hfill
+(\fieps'^2+(1\pm\fieps)^2)^{-3/2}((1\pm\fieps)^2+2\fieps'^2)\leqslant \frac{n+1}{1-b_\varepsilon}\hfill
}$$
and by differential Equation \eqref{equadif} we have
$$\displaylines{\hepsi(r)=\Bigl|(n-k-1)\frac{(\fieps'^2+(1\pm\fieps)^2)^{-1/2}}{1\pm\fieps}\fieps'\cot(r)+(\fieps'^2+(1\pm\fieps)^2)^{-3/2}(1\pm\fieps)\fieps''\Bigr|\cr
\leqslant(n-k-1)\frac{(\fieps'^2+(1\pm\fieps)^2)^{-1/2}}{1\pm\fieps}\fieps'\Bigl|\cot(r)-\frac{1}{r}\Bigr|\hfill\cr
+\frac{n-k-1}{r}\Bigl|\frac{(\fieps'^2+(1\pm\fieps)^2)^{-1/2}}{1\pm\fieps}\fieps'-(\fieps'^2+(1\pm\fieps)^2)^{-3/2} (1\pm\fieps)(1+\fieps'^2)\fieps'\Bigr|\cr
\leqslant\frac{n}{1-b_\varepsilon}\Bigl(\frac{1}{r}-\cot(r)\Bigr)\hfill\cr
+\frac{n\left(\fieps'^2+(1\pm\fieps)^2\right)^{-3/2}}{r(1\pm\fieps)}\fieps'\Bigl|\fieps'^2+(1\pm\fieps)^2-(1\pm\fieps)^2(1+\fieps'^2)\Bigr|\cr
\leqslant \frac{n}{1-b_\varepsilon}\Bigl(\frac{1}{r}-\cot(r)\Bigr)+\frac{n}{r}\fieps\frac{2\pm\fieps}{1\pm\fieps}\frac{\fieps'^3}{[\fieps'^2+(1\pm\fieps)^2]^{3/2}}\hfill\cr
\leqslant \frac{n}{1-b_\varepsilon}\Bigl(\frac{1}{r}-\cot(r)\Bigr)+\frac{n}{r}\fieps\frac{2+b_\varepsilon}{1-b\varepsilon}\hfill}$$
Since $\displaystyle\frac{\fieps}{r}=\frac{\varepsilon}{r}\int_1^{r/\varepsilon}\frac{dt}{\sqrt{t^{2(n-k-1)}-1}}\leqslant\frac{\varepsilon}{r}\int_1^{r/\varepsilon}\frac{dt}{\sqrt{t^{2}-1}}$ and $\frac{1}{x}\int_1^x\frac{dt}{\sqrt{t^2-1}}\sim_{+\infty}\frac{\ln x}{x}$, we get that $\hepsi$ is bounded on $M_\varepsilon^k$, hence $\H_\varepsilon$ is bounded on $M_\varepsilon$. By the Lebesgue theorem we have $\|\H_\varepsilon-1\|_1\to 0$. 

We now bound $\|\B _\varepsilon\|_q$ with $q=n-k$. The volume element at the neighbourhood of $M^+_\varepsilon\cap M^-_\varepsilon$ is 
\begin{equation}
  \label{vol}
  \voleps=(1\pm\fieps)^{n}(1+(\frac{\fieps'}{1\pm\fieps})^2)^{1/2}\sin^{n-k-1}(r)\cos^k(r) dv_{n-k-1}dv_k dr
\end{equation}
where $dv_{n-k-1}$ and $dv_k$ are the canonical volume element of $\S^{n-k-1}$ and $\S^k$ respectively. By Lemma \ref{courbmoy} and Equation \eqref{equadif}, we have 
$$\displaylines{|\B _\varepsilon|^q\voleps=\frac{1}{{(\coef)}^\frac{q}{2}}\max\Bigl(\bigl|1-\frac{\fieps'}{1\pm\fieps}\cot r\bigr|, \bigl|1+\frac{\fieps'}{1\pm\fieps}\tan r\bigr|,\hfill\cr
\hfill\bigl|1+\frac{\fieps'^2+(n-k-1)(1\pm\fieps)(1+\fieps'^2)\fieps'/r}{\coef}\bigr|\Bigr)\Bigr]^q \voleps}$$
Noting that $\frac{x}{\sqrt{1+x^2}}\leqslant\min(1,x)$, it is easy to see that, if we set $h_\varepsilon=\min(1,|\varphi'_\varepsilon|)$ 
\begin{align*}\frac{\bigl|1-\frac{\fieps'}{1\pm\fieps}\cot r\bigr|}{\sqrt{\coef}}&\leqslant\frac{1}{\sqrt{\fieps'^2+(1\pm\fieps)^2}}+\frac{\frac{\fieps'}{1\pm\fieps}}{\sqrt{\frac{\fieps'^2}{(1\pm\fieps)^2}+1}}\frac{\cot r}{1\pm\fieps}\\
 &\leqslant\frac{1}{1-\fieps}+\frac{h_{\varepsilon}\cot r}{(1-\fieps)^2}\leqslant4\left(1+\frac{h_{\varepsilon}}{r}\right)
\end{align*}
Similarly for $r\in[\varepsilon, \pi/5+\varepsilon]$ and $\varepsilon$ small enough, we have
$$\frac{\bigl|1+\frac{\fieps'}{1\pm\fieps}\tan r\bigr|}{\sqrt{\coef}}\leqslant 4(1+h_\varepsilon\tan r)\leqslant 8(1+h_\varepsilon r)\leqslant 8\bigl(1+\frac{h_\varepsilon}{r}\bigr)$$
And since $\fieps'=0$ for $r\geqslant\pi/5+\varepsilon$, this inequality is also true for $r\in(\varepsilon,\pi/2]$. Moreover
\begin{align*}&\frac{1}{\sqrt{\coef}}\Bigl|1+\frac{\fieps'^2+(n-k-1)(1\pm\fieps)(1+\fieps'^2)\fieps'/r}{\coef}\Bigr|\\
&\leqslant\frac{1}{1\pm\fieps}+\frac{\fieps'^2}{(\coef)^{3/2}}+\frac{n}{r}\frac{(1\pm\fieps)(1+\fieps'^2)}{\coef}\frac{|\fieps'|}{(\coef)^{1/2}}\\
&\leqslant\frac{2}{1\pm\fieps}+\frac{nh_\varepsilon}{r(1-\fieps)}\frac{(1\pm\fieps)(1+\fieps'^2)}{\coef}\leqslant\frac{2}{1\pm\fieps}+2\frac{nh_\varepsilon}{r}\frac{(1+\fieps)^2}{(1-\fieps)^2}\\
&\leqslant 2\bigl(2+9\frac{nh_\varepsilon}{r}\bigr)
\end{align*}
It follows that
\begin{align*}
|\B _{\varepsilon}|^q\voleps&\leqslant C(n,k)\bigl(1+\frac{h_\varepsilon}{r}\bigr)^q\voleps\leqslant C(n,k)(r+h_\varepsilon)^q r^{-1}\bigl(1+\frac{\fieps'}{1\pm\fieps}\bigr)dv_{n-k-1}dv_{k}dr\\
&\leqslant C(n,k)r^{-1}(r+h_\varepsilon)^q\Bigl(1+\frac{1}{\sqrt{(r/\varepsilon)^{2(n-k-1)}-1}}\Bigr)dv_{n-k-1}dv_kdr\end{align*}
Now
$$\displaylines{\int_{\meps^k}|\B _\varepsilon|^q\voleps\leqslant C(n,k)\Bigl(\int_\varepsilon^{2^\frac{1}{2(n-k-1)}\varepsilon}r^{-1}\Bigl(1+\frac{1}{\sqrt{(r/\varepsilon)^{2(n-k-1)}-1}}\Bigr)dr\hfill\cr
\hfill+\int_{2^\frac{1}{2(n-k-1)}\varepsilon}^{2a+\varepsilon} r^{n-k-1}\Bigl(1+\frac{1}{r\sqrt{(r/\varepsilon)^{2(n-k-1)}-1}}\Bigr)^qdr\Bigr)\cr
\leqslant C(n,k)\Bigl(\int_1^{2^\frac{1}{2(n-k-1)}}s^{-1}\Bigl(1+\frac{1}{\sqrt{s^{2(n-k-1)}-1}}\Bigr)ds+\int_{2^\frac{1}{2(n-k-1)}}^{2a/\varepsilon+1} s^{n-k-1}\bigl(\varepsilon+\frac{1}{s^{q}}\bigr)^qds\Bigr)\cr\\}
$$
Since $\varepsilon^\frac{-1}{q}\leqslant\frac{2a}{\varepsilon}+1$ for $\varepsilon$ small enough we have
$$\displaylines{\int_{\meps^k}|\B _\varepsilon|^q\voleps\leqslant C(n,k)\Bigl(1+\int_{2^\frac{1}{2(n-k-1)}}^{\varepsilon^\frac{-1}{q}} \frac{2s^{n-k-1}}{s^{q^2}}ds+\int_{\varepsilon^\frac{-1}{q}}^{2a/\varepsilon+1} 2s^{n-k-1}\varepsilon^qds\Bigr)\cr
\leqslant C(n,k)\bigl(1+\varepsilon^{n-k-1}\bigr)\\}$$
which remains bounded when $\varepsilon\to0$.
\end{proof}

Since $\varphi_\varepsilon$ is constant outside a neighborhood of $M_\varepsilon^+\cap M_\varepsilon^-$ (given by $a$), $M^k_\varepsilon$ is a smooth submanifold diffeomorphic to the sum of two spheres $\S^n$ along a (great) subsphere $\S^k\subset\S^n$.
\begin{center}
\includegraphics[width=1.5cm]{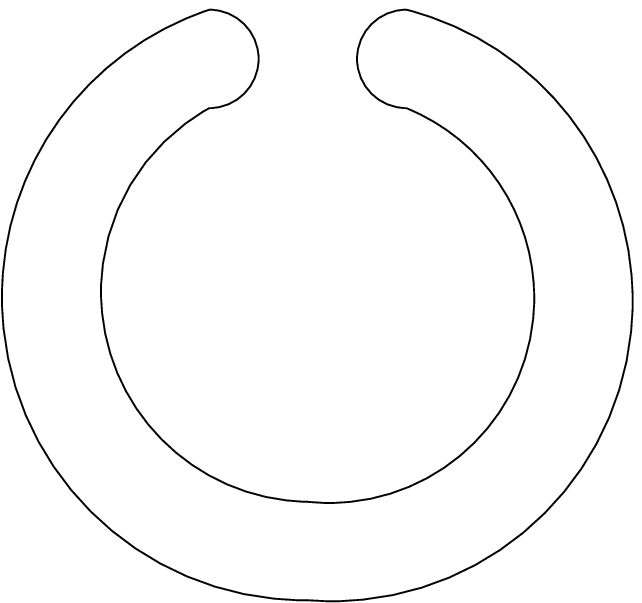}
\end{center}
If we denote $\tilde{M}^k_\varepsilon$ one connected component of the points of $M^k_\varepsilon$ corresponding to $r\leqslant 3a$, we get some pieces of hypersurfaces
\begin{center}
\includegraphics[width=1cm]{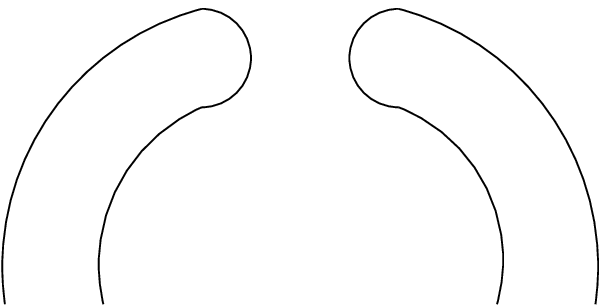}
\end{center}
that can be glued together along pieces of spheres of constant curvature to get a smooth submanifold $M_\varepsilon$, diffeomorphic to $p$ spheres $\S^n$ glued each other along $l$ subspheres $S_i$, and with curvature satisfying the bounds of Theorem \ref{ctrexple1} (when all the subspheres have dimension $0$) or of the remark before Theorem \ref{Lipschitz}.
\begin{center}
\includegraphics[width=2cm]{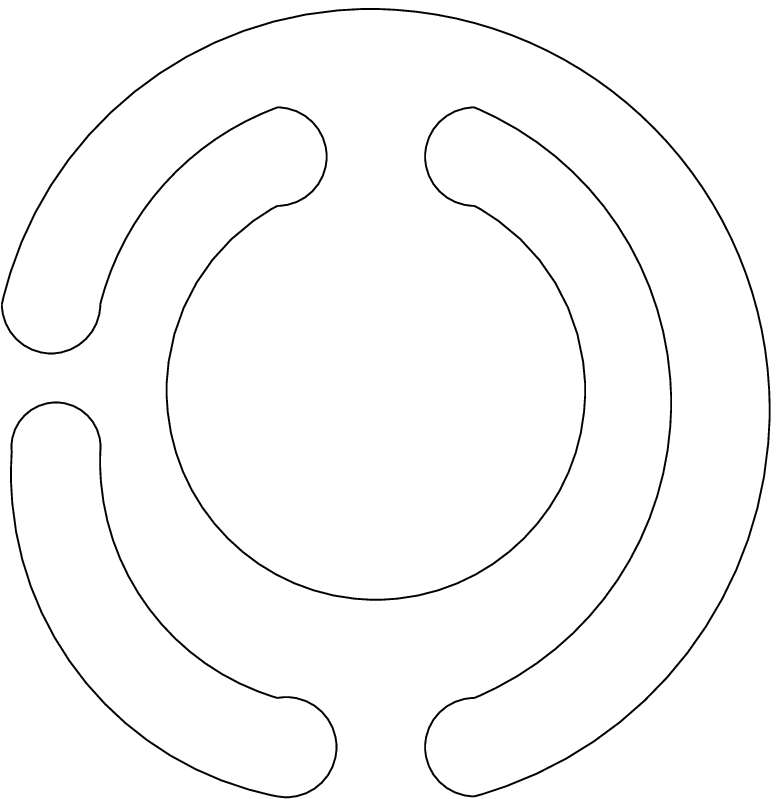}
\end{center}
Since the surgeries are performed along subsets of capacity zero, the manifold constructed have a spectrum close to the spectrum of $p$ disjoints spheres of radius close to $1$ (i.e. close to the spectrum of the standard $\S^n$ with all multiplicities multiplied by $p$).
More precisely, we set $\eta\in[2\varepsilon,\frac{\pi}{20}]$, and for any subsphere $S_i$, we set $N_{i,\eta,\varepsilon}$ the tubular neighborhood of radius $\eta$ of the submanifold $\tilde{S}_i=M_{\varepsilon,i}^+\cap M_{\varepsilon,i}^-$ in the local parametrization of $M_\varepsilon$ given by the map $\Phi_{\varphi_{\varepsilon,i}}$ associated to the subsphere $S_i$. We have $\meps=\Omega_{1,\eta,\varepsilon}\cup\cdots\cup \Omega_{p,\eta,\varepsilon}\cup N_{1,\eta, \varepsilon}\cup \cdots\cup N_{l,\eta,\varepsilon}$ where $\Omega_{i,\eta,\varepsilon}$ are the connected component of $M\setminus\cup_i N_{i,\eta,\varepsilon}$. The $\Omega_{i,\eta,\varepsilon}$ are diffeomorphic to some $S_{i,\eta}$ (which does not depend on $\varepsilon$ and $\eta$) open set of $\S^n$ which are complements of neighborhoods of subspheres of dimension less than $n-2$ and radius $\eta$, endowed with metrics which converge in $\mathcal{ C}^1$ topology to standard metrics of curvature $1$ on $S_{i,\eta}$. Indeed, $\varphi_\varepsilon$ converge to 0 in topology $\mathcal{ C}^2$ on $[r^{i,\pm}_{\varepsilon,\eta},\frac{\pi}{2}]$, where $\int_\varepsilon^{r_{\varepsilon,\eta}^{i,\pm}}\sqrt{(1\pm\varphi_{\varepsilon,i})^2+(\varphi_{\varepsilon,i}')^2}=\eta$ since it converges in $\mathcal{ C}^1$ topology on any compact of $[\varepsilon,\frac{\pi}{2}]$ and since we have
$$\displaylines{\eta\geqslant\int_\varepsilon^{r^{i,\pm}_{\varepsilon,\eta}}(1-b_{i,\varepsilon})\,dt=(r^{i,\pm}_{\varepsilon,\eta}-\varepsilon)(1-b_{i,\varepsilon})
\cr
\eta\leqslant\int_\varepsilon^{r^{i,\pm}_{\varepsilon,\eta}}(1+b_{i,\varepsilon})\,dt+\int_\varepsilon^{r^{i,\pm}_{\varepsilon,\eta}}\frac{dt}{\sqrt{(\frac{t}{\varepsilon})^{2(n-k-1)}-1}}=(r^{i,\pm}_{\varepsilon,\eta}-\varepsilon)(1+b_{i,\varepsilon})\cr\hspace{7cm}+\varepsilon\int_1^{+\infty}\frac{dt}{\sqrt{t^{2(n-k-1)}-1}}}$$
so $r_{\varepsilon,\eta}^\pm\to\eta$ when $\varepsilon\to 0$.
 So the spectrum of $\cup_i\Omega_{i,\eta,\varepsilon}\subset M_\varepsilon$ for the Dirichlet problem converges to the spectrum of $\amalg_i S_{i,\eta}\subset \amalg_i\S^n$ for the Dirichlet problem as $\varepsilon$ tends to $0$ (by the min-max principle). Since any subsphere of codimension at least $2$ has zero capacity in $\S^n$, we have that  the spectrum of $\amalg_i S_{i,\eta}\subset\amalg_i \S^n$ for the Dirichlet problem converges to the spectrum of $\amalg_i\S^n$ when $\eta$ tends to $0$ (see for instance \cite{Co} or adapt what follows). Since the spectrum of $\amalg_i\S^n$ is the spectrum of $\S^n$ with all multiplicities multiplied by $p$, by diagonal extraction we infer the existence of two sequences $(\varepsilon_m)$ and $(\eta_m)$ such that $\varepsilon_m\to0$, $\eta_m\to 0$ and  the spectrum of $\cup_i\Omega_{i,\eta_m,\varepsilon_m}\subset M_{\varepsilon_m}$ for the Dirichlet problem converges to the spectrum of $\S^n$ with all multiplicities multiplied by $p$.
Finally, note that $\lambda_l(M_\varepsilon)\leqslant\lambda_l(\cup_i\Omega_{i,2\eta,\varepsilon})$ for any $l$ by the Dirichlet principle. 

On the other hand, by using functions of the distance to the $\tilde{S}_i$ we can easily construct on $M_\varepsilon$ a function $\psi_\varepsilon$ with value in $[0,1]$, support in $\cup_i\Omega_{i,\eta,\varepsilon}$, equal to $1$ on $\cup_i\Omega_{i,2\eta,\varepsilon}$ and whose gradient satisfies $|d \psi_\varepsilon|_{g_\varepsilon}^{~}\leqslant\frac{2}{\eta}$. It readily follows that
$$\|1-\psi_\varepsilon^2\|_1^{~}+\|d\psi_\varepsilon\|^{2}_2\leqslant(1+\frac{4}{\eta^2})\sum_i\frac{\Vol N_{i,2\eta,\varepsilon}}{\Vol M_\varepsilon}$$
To estimate $\sum_i\Vol N_{i,2\eta,\varepsilon}$, note that $N_{i,2\eta,\varepsilon}$ corresponds to the set of points with $r^{i,\pm}\leqslant r^{i,\pm}_{\varepsilon,2\eta}$ in the parametrization of $M_\varepsilon$ given by $\Phi_{\varphi_{\varepsilon,i}}$ at the neighborhood of $\tilde{S}_i$, where, as above, $r^{i,\pm}_{\varepsilon,2\eta}$ is given by
$$\int_\varepsilon^{r^{i,\pm}_{\varepsilon,2\eta}}\sqrt{(1\pm\varphi_{\epsilon,i})^2+(\varphi_{\epsilon,i}')^2}=2\eta$$
hence satisfies $\frac{1}{2}(r^{i,\pm}_{\varepsilon,2\eta}-\varepsilon)\leqslant 2\eta$ (since we have $1-\varphi_{\varepsilon,i}\geqslant\frac{1}{2}$).
By formula \ref{vol}, we have 
$$\displaylines{\Vol N_{i,2\eta,\varepsilon}\leqslant C(n)\int_\varepsilon^{r^-_\eta}(1-\varphi_{\varepsilon,i})^{n-1}\sqrt{(1-\varphi_{\varepsilon,i})^2+(\varphi_{\varepsilon,i}')^2}t^{n-k-1}dt\hfill\cr
\hfill+ C(n)\int_\varepsilon^{r^+_\eta}(1+\varphi_{\varepsilon,i})^{n-1}\sqrt{(1+\varphi_{\varepsilon,i})^2+(\varphi_{\varepsilon,i}')^2}t^{n-k-1}dt\hfill\cr
\hfill\leqslant C(n)(4\eta+\varepsilon)^{n-k-1}\eta\leqslant C(n,k)\eta^{n-k}}$$
where we have used that $\varphi_{\varepsilon,i}\leqslant2$ and $2\varepsilon\leqslant\eta$. We then have
$$\|1-\psi_\varepsilon^2\|_1^{~}+\|d\psi_\varepsilon\|_2^2\leqslant C(n,k,l,p)\eta^{n-k}$$
To end the proof of the fact that $M_{\varepsilon_m}$ has a spectrum close to that of $\cup_i\Omega_{i,\eta_m,\varepsilon_m}$ we need the following proposition, whose proof is a classical Moser iteration (we use the Simon and Michael Sobolev Inequality).
\begin{proposition}\label{norminffoncprop}  For any $q>n$ there exists a constant $C(q,n)$ so that if $(M^n,g)$ is any Riemannian manifold isometrically immersed in $\R^{n+1}$ and $E_N=\langle f_0,\cdots,f_N\rangle$is the space spanned by the eigenfunctions associated to $\lambda_0\leqslant\cdots\leqslant\lambda_N$, then for any $f\in E_N$ we have $$\|f\|_{\infty}\leqslant C(q,n)\left((v_M)^{1/n}(\lambda_N^{1/2}+\|\H\|_q)\right)^{\gamma}\|f\|_2$$
where $\gamma=\frac{1}{2}\frac{qn}{q-n}$.  
\end{proposition}
Since we already know that $\lambda_\sigma(M_{\varepsilon_m})\leqslant\lambda_\sigma(\cup_i\Omega_{i,\eta_m,\varepsilon_m})\to\lambda_{E(\sigma/p)}(\S^n)$ for any $\sigma$ when $m\to\infty$, we infer that for any $N$ there exists $m=m(N)$ large enough such that on $M_{\varepsilon_m}$ and for any $f\in E_N$, we have (with $q=2n$ and since $\|\H\|_\infty\leqslant C(n)$)
$$\|f\|_{\infty}\leqslant C(p,N,n)\|f\|_2$$
By the previous estimates, if we set
$$L_{\varepsilon_m}:f\in E_N\mapsto\psi_{\varepsilon_m} f\in \H^1_0(\cup_i\Omega_{i,\eta_m,\varepsilon_m})$$
then we have
$$\|f\|_2^2\geqslant\|L_{\varepsilon_m}(f)\|_2^2\geqslant\|f\|_2^2-\|f\|_\infty^2\|1-\psi_{\varepsilon_m}^2\|_1\geqslant\|f\|_2^2\bigl(1-C(k,l,p,N,n)\eta_m^{n-k}\bigr)$$
and 
$$\displaylines{\|dL_{\varepsilon_m}(f)\|_2^2=\frac{1}{\Vol M_{\varepsilon_m}}\int_{M_{\varepsilon_m}}|fd\psi_{\varepsilon_m}+\psi_{\varepsilon_m} df|^2\hfill\cr
\leqslant(1+h)\|df\|^2_2+(1+\frac{1}{h})\frac{1}{\Vol M_{\varepsilon_m}}\int_{M_{\varepsilon_m}}f^2|d\psi_{\varepsilon_m}|^2\cr
\hfill\leqslant(1+h)\|df\|^2_2+(1+\frac{1}{h})C(k,l,p,N,n)\|f\|_2^2\eta_m^{n-k}}$$
for any $h>0$. We set $h=\eta_m^\frac{n-k}{2}$. For $m=m(k,l,p,N,n)$ large enough, $L_{\varepsilon_m}:E_N\to \H^1_0(\cup_i\Omega_{i,\eta_m,\varepsilon_m})$ is injective and for any $f\in E_N$, we have
$$\frac{\|dL_{\varepsilon_m}(f)\|_2^2}{\|L_{\varepsilon_m}(f)\|_2^2}\leqslant (1+C(k,l,p,N,n)\eta_m^\frac{n-k}{2})\frac{\|df\|_2^2}{\|f\|_2^2}+C(k,l,p,N,n)\eta_m^\frac{n-k}{2}$$
By the min-max principle, we infer that for any $\sigma\leqslant N$, we have
$$\lambda_\sigma(M_{\varepsilon_m})\leqslant\lambda_\sigma(\cup_i\Omega_{i,\eta_m,\varepsilon_m})\leqslant(1+C(k,l,p,N,n)\eta_m^\frac{n-k}{2})\lambda_\sigma(M_{\varepsilon_m})+C(k,l,p,N,n)\eta_m^\frac{n-k}{2}$$
Since $\lambda_\sigma(\cup_i\Omega_{i,\eta_M,\varepsilon_m})\to \lambda_{E(\sigma/p)}(\S^n)$, this gives that $\lambda_\sigma(M_{\varepsilon_m})\to \lambda_{E(\sigma/p)}(\S^n)$ for any $\sigma\leqslant N$. By diagonal extraction we get the sequence of manifolds $(M_j)$ of Theorem \ref{ctrexple1}.

To construct the sequence of Theorem \ref{ctrexple3}, we consider the sequence of embedded submanifolds $(M_j)$ of Theorem \ref{ctrexple1} for $p=2$, $k=n-2$ and $l=1$. Each element of the sequence admits a covering of degree $d$ given by $y\mapsto y^d$ in the  local charts associated to the maps $\Phi$. We endow these covering with the pulled back metrics. Arguing as above, we get that the spectrum of the new sequence converge to the spectrum of two disjoint copies of 
$$\bigl(\S^1\times\S^{n-2}\times[0,\frac{\pi}{2}],dr^2+d^2\sin^2rg_{\S^1}+\cos^2rg_{\S^{n-2}}\bigr).$$

\appendix
\section{Proof of Lemma \ref{courbmoy}}

Let $(u,v,h)\in T_x S_{\varepsilon}$ and put $w=d(\Phi_{\varphi})_x (u,v,h)\in T_q X_{\varphi}$ where $S_{\varepsilon}=\S^{n-k-1}\times\S^{k}\times \ieps$. An easy computation shows that 
\begin{align}\label{difphi}w&=(1+\varphi(r))((\sin r) u+(\cos r) v)\notag\\
&+\varphi'(r)((\sin r) y+(\cos r) z)h+(1+\varphi(r))((\cos r)y-(\sin r)z)h\end{align}
We set 
$$\tilde{N}_q=-\varphi'(r)((\cos r) y-(\sin r) z)+(1+\varphi(r))((\sin r) y+(\cos r) z)$$
and $\displaystyle N_q=\frac{\tilde{N}_q}{\left(\coeff\right)^{1/2}}$ is a unit normal vector field on $X_{\varphi}$.
Then we have
\begin{align}\label{defsec} \B _q(\varphi)(w,w)&=\left\langle\nabla_{w}^0 N,w\right\rangle\notag=\left( \coeff\right)^{-1/2}\bigl\langle\nabla_{w}^0\tilde{N},w\bigr\rangle\notag\\
&=\left( \coeff\right)^{-1/2}\Bigl\langle\sum_{i=1}^{n+1} w(\tilde{N}^i)\partial_i,w\Bigr\rangle\end{align}
where $(\partial_i)_{1\leqslant i\leqslant n+1}$ is the canonical basis of $\R^{n+1}$. A straightforward computation shows that
\begin{align*}\sum_{i=1}^{n+1}w(\tilde{N}^i)\partial_i=&-\varphi'(r)((\cos r) u-(\sin r) v)+(1+\varphi(r))((\sin r) u+(\cos r) v)\\
&-\varphi''(r)((\cos r) y-(\sin r) z) h+2\varphi'(r)((\sin r) y+(\cos r) z)h\\
&+(1+\varphi(r))((\cos r) y-(\sin r) z) h\end{align*}
Reporting this in (\ref{defsec}) and using (\ref{difphi}) we get 
$$\displaylines{\B _q(\varphi)((u,v,h),(u,v,h))=\frac{1}{\sqrt{\coeff}}\Bigl[-\varphi'(r)\bigl(1+\varphi(r)\bigr)\sin r \cos r (|u|^2-|v|^2)\cr
\hfill+(1+\varphi(r))^2 (\sin^2 r |u|^2 +\cos^2 r |v|^2) -\bigl(1+\varphi(r)\bigr)\varphi''(r)h^2+2\varphi'^2(r)h^2+(1+\varphi(r))^2h^2\Bigr]}$$ 
Now let $(u_i)_{1\leqslant i\leqslant n-k-1}$ and $(v_i)_{1\leqslant i\leqslant k}$ be orthonormal bases of respectively $\S^{n-k-1}$ at $y$ and $\S^{k}$ at $z$. We set $g=\Phi_{\varphi}^{\star} can$ and $\xi=(0,0,1)$, then we have
$$\displaylines{
\hfill g(u_i,u_j)=(1+\varphi(r))^2\sin^2 r\delta_{ij},\hfill g(v_i,v_j)=(1+\varphi(r))^2\cos^2 r\delta_{ij}, \hfill g(u_i,v_j)=0,\hfill\cr
\hfill g(\xi,\xi)=\coeff,\hfill g(u_i,\xi)=g(v_j,\xi)=0.\hfill}$$
Now setting $\tilde{u}_i=d(\Phi_{\varphi})_x(u_i)$, $\tilde{v}_i=d(\Phi_{\varphi})_x(u_i)$ and  $\tilde{\xi}=d(\Phi_{\varphi})_x(\xi)$, the relation above allows us to compute the trace and norm
\begin{align*}
&|\B_{q}(\varphi)|=\max\Bigl(\max_{i}\frac{|\B_q(\varphi)(\tilde{u}_i,\tilde{u}_i)|}{g(u_i,u_i)},\max_j\frac{|\B_q(\varphi)(\tilde{v}_j,\tilde{v}_j)|}{g(v_j,v_j)},\frac{|\B_q(\varphi)(\tilde{\xi},\tilde{\xi})|}{g(\xi,\xi)}\Bigl)\\
&{=}\frac{1}{\sqrt{\coeff}}\max\Bigl(\bigl|1{-}\frac{\varphi '}{1{+}\varphi}\cot r\bigr|,\bigl|1+\frac{\varphi '}{1{+}\varphi}\tan r\bigr|,\bigl|1{+}\frac{(\varphi')^2-(1+\varphi)\varphi''}{\coeff}\bigr|\Bigr)\end{align*}
of the second fundamental form.

\end{document}